\documentclass[reqno,12pt]{amsart}
\usepackage{geometry}
\usepackage{amsmath,amssymb,mathrsfs,color}
\usepackage[upint]{stix}
\usepackage[colorlinks,
linkcolor=red,
anchorcolor=green,
citecolor=blue,
]{hyperref}
\usepackage{subfigure}
\usepackage{float}
\usepackage{times}
\usepackage{tikz}
\usetikzlibrary{intersections}
\usepackage{paralist}

\usepackage{comment}

\makeatletter
\def\setliststart#1{\setcounter{\@listctr}{#1}%
  \addtocounter{\@listctr}{-1}}
\makeatother

\makeatletter
\@addtoreset{figure}{section}
\makeatother

\usepackage{calc}
 \newtheorem{The}{Theorem}[section]
 \newtheorem{Cor}[The]{Corollary}
 \newtheorem{Lem}[The]{Lemma}
 \newtheorem{Pro}[The]{Proposition}
 \theoremstyle{definition}
 \newtheorem{defn}[The]{Definition}
 \newtheorem{Rem}[The]{Remark}
 
 \numberwithin{equation}{section}

\newcommand{\R}{\mathbb{R}}

\newcommand{\N}{\mathbb{N}}

\newcommand{\SING}{\mbox{\rm Sing}\,(u)}

\newcommand{\argmin}{\operatorname*{arg\ min}}
\newcommand{\supp}{\mbox{\rm supp}\,}
\newcommand{\singsupp}{\mbox{\rm sing\ supp}\,}

\title[Local strict singular characteristics I]{Local strict singular characteristics: \\Cauchy problem with smooth initial data}
\author{Wei Cheng \and Jiahui Hong}
\address{Department of Mathematics, Nanjing University, Nanjing 210093, China}
\email{chengwei@nju.edu.cn}
\address{Department of Mathematics, Nanjing University, Nanjing 210093, China}
\email{hjh9413@163.com}
\date{\today}
\begin{document}
\maketitle


\begin{abstract}
	Main purpose of this paper is to study the local propagation of singularities of  viscosity solution to contact type evolutionary Hamilton-Jacobi equation
	\begin{align*}
		D_tu(t,x)+H(t,x,D_xu(t,x),u(t,x))=0.
	\end{align*}
	An important issue of this topic is the existence, uniqueness and regularity of the strict singular characteristic. We apply the recent existence and regularity results on the Herglotz' type variational problem to the aforementioned Hamilton-Jacobi equation with smooth initial data. We obtain some new results on the local structure of the cut set of the viscosity solution near non-conjugate singular points. Especially, we obtain an existence result of smooth strict singular characteristic from and to non-conjugate singular initial point based on the structure of the superdifferential of the solution, which is even new in the classical time-dependent case. We also get a global propagation result for the $C^1$ singular support in the contact case.
\end{abstract}

\section{Introduction}

Main purpose of this paper is twofold. First, we want to study the Bolza problem of Herglotz' type. We will emphasize on the Jacobi condition and the structure of the cut locus with respect to the associated value function. Second, we will analyze the propagation of singularities from and to non-conjugate points along smooth strict singular characteristics firstly touched in \cite{Khanin_Sobolevski2016} (see also \cite{Cannarsa_Cheng2020}). 

Throughout this paper, we assume $L=L(t,x,v,r):\R\times\R^n\times\R^n\times\R\rightarrow\R$ is a function of class $C^{R+1}$ ($R\geqslant 1$) such that the following standing assumptions are satisfied:
\begin{enumerate}[(L1)]
  \item $L(t,x,\cdot,r)$ is strictly convex for all $(t,x,r)\in \R\times\R^n\times\R$.
  \item There exist two superlinear functions $\overline{\theta}_0,\theta_0:[0,+\infty)\to[0,+\infty)$ and two $L^{\infty}_{\rm loc}$-functions $c_0,\bar{c}_0:\R\to[0,+\infty)$, such that
  \begin{align*}
  	\overline{\theta}_0(|v|)+\bar{c}_0(t)\geqslant L(t,x,v,0)\geqslant\theta_0(|v|)-c_0(t),\quad (t,x,v)\in\R\times\R^n\times\R^n\times\R.
  \end{align*}
  \item There exists an $L^{\infty}_{\rm loc}$-function $K:\R\to[0,+\infty)$ such that
  \begin{align*}
  	|L_r(t,x,v,r)|\leqslant K(t),\quad (t,x,v,r)\in\R\times\R^n\times\R^n\times\R.
  \end{align*}
  \item There exists two $L^{\infty}_{\rm loc}$-functions $C_1,C_2:\R\to[0,+\infty)$ such that
  \begin{align*}
  	|L_t(t,x,v,r)|\leqslant C_1(t)+C_2(t)L(t,x,v,r),\quad (t,x,v,r)\in\R\times\R^n\times\R^n\times\R.
  \end{align*}
\end{enumerate}

Given any $x,y\in\R^n$, $a<b$ and $u\in\R$, we denote by $\Gamma^{a,b}_{x,y}$ the set of absolutely continuous functions $\xi\in W^{1,1}([a,b],\R^n)$ such that $\xi(a)=x$ and $\xi(b)=y$. The variational problem of Herglotz' type is to solve
\begin{align*}
	\inf_{\xi\in\Gamma^{a,b}_{x,y}}\int^{b}_{a}L(s,\xi(s),\dot{\xi}(s),u_{\xi}(s))\ ds,
\end{align*}
where $u_{\xi}$ is uniquely determined by the \emph{Carath\'eodory equation}
\begin{align*}
	\begin{cases}
		\dot{u}_{\xi}(s)=L(s,\xi(s),\dot{\xi}(s),u_{\xi}(s)),\quad s\in[a,b],&\\
		u_{\xi}(t_1)=u.&
	\end{cases}
\end{align*}
The existence and regularity issues of this problem were solved in \cite{CCWY2019,CCJWY2020} in a rigorous way recently.

The first part of this paper is composed of a collection of results on the Bolza problem of Herglotz' type, and relevant results on the viscosity solutions of the associated Hamilton-Jacobi equation
\begin{equation}\label{eq:HJe}\tag{HJ$_e$}
	\left\{
	 \begin{split}
	 	D_tu(t,x)+H(t,x,D_xu(t,x),u(t,x))=&\,0\\
	 	u(0,x)=&\,u_0(x)
	 \end{split}
	 \right.\quad x\in \R^n, t>0,
\end{equation}
with $u_0$ of class $C^{R+1}$. Recall that the Hamiltonian $H:\R\times\R^n\times\R^n\times\R\rightarrow\R$ is defined by
\begin{align*}
	H(t,x,p,r)=\sup_{v\in\R^n}\{\langle p,v\rangle-L(t,x,v,r)\},\qquad (t,x,p,r)\in\R\times\R^n\times\R^n\times\R.
\end{align*}
For other variational approach of equation \eqref{eq:HJe}, see also \cite{Wang_Yan2019,Wang_Wang_Yan2017} and references therein.

We give the definition of \emph{conjugate points} and \emph{irregular points} for this Herglotz-type problem and clarify the structure of the cut locus $\bar{\Sigma}$. A very important observation (Proposition \ref{irregu non-conj pro}) is that, if $(\bar{t},\bar{x})$ is not conjugate then $u$ has a local representation as the minimum of a finite family of smooth function, i.e., there exists $r_0>0$ such that
\begin{equation}\label{eq:intro_finite_repre}
	u(t,x)=\min_{i=1,\ldots,k}v_i(t,x),\quad (t,x)\in B_{r_0}((\bar{t},\bar{x}))
\end{equation}
where all $v_i's$ are of class $C^{R+1}$. This is a standing point of our sequel analysis of the propagation of singularities around a non-conjugate point.

It is well known that Hamilton-Jacobi equations have no global smooth solutions in general, because solutions may develop singularities due to crossing or focusing of characteristics. The persistence of singularities, i.e, once a singularity is created, it will propagate forward in time up to $+\infty$. The expected maximal regularity for solutions of \eqref{eq:HJe} is the local semiconcavity of $u$. See, for instance, \cite{Cannarsa_Sinestrari_book} and \cite{Villani_book2009} for more details for the notion of semiconcavity.

In the seminal paper \cite{Albano_Cannarsa2002}, Albano and Cannarsa introduced the important notion of \emph{generalized characteristics} for Hamilton-Jacobi equation \eqref{eq:HJe}, which is a keystone for the study of the problem of propagation of singularities later. In one-dimensional case, the idea of generalized characteristics  also comes from earlier work by Dafermos \cite{Dafermos1977} on Burgers equation. Recall that a Lipschitz curve $\mathbf{x}:[0,T]\to\Omega$, $\mathbf{x}(0)=x_0\in\Sigma$, the set of non-differentiability of $u$, is called a generalized characteristic from $x_0$ with respect to the Hamilton-Jacobi equation
\begin{equation}
	H(x,Du(x),u(x))=0,\qquad x\in\Omega,
\end{equation}
if the following differential inclusion is satisfied
\begin{equation}\label{eq:intro_gc}
	\dot{\mathbf{x}}(t)\in\text{co}\,H_p(\mathbf{x}(t),D^+u(\mathbf{x}(t)),u(\mathbf{x}(t))),\quad a.e.,\ t\in[0,T].
\end{equation}
Local structure of generalized characteristics was systematically studied in \cite{Cannarsa_Yu2009}. 

However, the convex hull in \eqref{eq:intro_gc} is an obvious obstacle to establish the well-posedness of the differential inclusion \eqref{eq:intro_gc} such as uniqueness and stability. Khanin and  Sobolevski's celebrating results (\cite{Khanin_Sobolevski2016}) established the existence of the singular characteristics satisfies \eqref{eq:intro_gc} without convex hull under some extra conditions on the initial data for classical time-dependent Hamiltonians (see also \cite{Stromberg_Ahmadzadeh2014} for some relevant discussion). Given a locally semiconcave solution $u$ of the equation
\begin{align*}
	\begin{cases}
		D_tu(t,x)+H(t,x,D_xu(t,x))=0\\
	 	u(0,x)=u_0(x),
	\end{cases}
	\quad x\in \R^n, t>0,
\end{align*}
a Lipschitz singular curve $\mathbf{x}:[t_0,t_0+\delta]\to(0,+\infty)\times\R^n$ is called a \emph{strict singular characteristic} from $(t_0,x_0)\in\SING$ if there exists a right continuous selection $p(t)\in D^+u(t,\mathbf{x}(t))$ such that
\begin{equation}\label{eq:intro_sgc}
	\begin{split}
		\begin{cases}
		\dot{\mathbf{x}}^+(t)=H_p(t,\mathbf{x}(t),p(t))& \forall\ t\in [t_0,t_0+\delta],\\
		\mathbf{x}(t_0)=x_0,&
	\end{cases}
	\end{split}
\end{equation}
where the curve $p(\cdot)$ satisfies the following energy condition
\begin{equation}\label{eq:KS_energy}
	H(t,\mathbf{x}(t),p(t))=\min_{p\in D^+u(t,\mathbf{x}(t))}H(t,\mathbf{x}(t),p)\qquad\forall t\in[t_0,t_0+\delta].
\end{equation}

In a recent paper (\cite{Cannarsa_Cheng2020}), the authors showed that if the initial point is not a critical point with respect to $(H,u)$, then all the singular characteristics is unique up to a bi-Lipschitz homeomorphism and the strict singular characteristic has uniqueness, for 2D stationary equation. However, the strict singular characteristics are not well understood since lack of more information from the underlying characteristic systems, comparing to the intrinsic approach of global propagations in \cite{Cannarsa_Cheng3}, \cite{Cannarsa_Cheng_Fathi2017} and \cite{Cannarsa_Cheng_Fathi2019}.

In this paper, we solve this problem for Hamilton-Jacobi equation \eqref{eq:HJe} around a non-conjugate singular point. Under certain non-degenerate condition we proved that the strict singular characteristic is locally a \emph{smooth} curve with uniqueness for the problem in arbitrary dimension. Moreover, we characterize such strict singular characteristics in a way with high geometric intuition (see Theorem \ref{thm:local propagation positive}). We also obtained a smooth strict singular characteristic approach a \emph{minimax} non-conjugate singular point (see Theorem \ref{thm:local propagation negative}). These results also lead to a clear picture of the singular set $\Sigma$ near a non-conjugate singular point at least when $n=1$.

We also give a global result on the propagation of the $C^1$ singular support of viscosity solutions of contact type Hamilton-Jacobi equation \eqref{eq:HJe} (Theorem \ref{global propagation}). This result is known for classic time-dependent one (see \cite{Albano2014_1}).

The paper is organized as follows: In Section 2 and 3, we introduce the Bolza problem of Herglotz type and obtain a collection of results on the Jacobi conditions and the structure of the cut set of the value functions which are viscosity solutions of \eqref{eq:HJe}. In Section 4, we analyze the strict singular characteristics around the non-conjugate singular points. In Section 5, we prove the global propagation of the $C^1$-singular support. The last section is a collection of the proofs of all the statements in Section 2.

\medskip

\noindent\textbf{Acknowledgements.} 
Wei Cheng is partly supported by National Natural Science Foundation of China (Grant No. 11871267, 11790272 and 11631006). 

\section{Bolza problem of Herglotz' type}\label{section_3_1}

In this section, we are devoted to study the following Bolza problem of Herglotz type: for any $(t,x)\in(0,+\infty)\times\R^n$,
\begin{equation}\label{eq:cov}\tag{COV$_{t,x}$}
	\inf_{\xi\in\mathcal{A}_{t,x}}\{u_0(\xi(0))+\int^{t}_{0}L(s,\xi(s),\dot{\xi}(s),u_{\xi}(s))\ ds\}
\end{equation}
where $\mathcal{A}_{t,x}$ is the set of absolutely continuous curve $\xi:[0,t]\to\R^n$ such that $\xi(t)=x$, and $u_{\xi}:[0,t]\to\R$ is a curve uniquely determined by the \emph{Carath\'eodory equation}
\begin{equation}\label{eq:cara1}
	\begin{cases}
		\dot{u}_{\xi}(s)=L(s,\xi(s),\dot{\xi}(s),u_{\xi}(s)),\quad s\in[0,t],&\\
		u_{\xi}(0)=u_0(\xi(0)),&
	\end{cases}
\end{equation}
with $u_0$ of class $C^{R+1}$ and bounded below by a function $(\kappa_1,\kappa_2)$-Lipschitz in the large\footnote{Let $(x,d)$ be a metric space. A function $\phi:X\to\R$ is called $(\kappa_1,\kappa_2)$-Lipschitz in the large if there exists $\kappa_1,\kappa_2\geqslant 0$ such that $|\phi(y)-\phi(x)|\leqslant\kappa_1+\kappa_2d(x,y)$ for all $x,y\in X$.}.

For any $t_2>t_1$, $x,y\in\R^n$ and $u\in\R$, we define
\begin{align*}
	h_L(t_1,t_2,x,y,u):=\inf_{\xi\in\Gamma^{a,b}_{x,y}}\int^{t_2}_{t_1}L(s,\xi(s),\dot{\xi}(s),u_{\xi}(s))\ ds,
\end{align*}
where $u_{\xi}$ is uniquely determined by the \emph{Carath\'eodory equation}
\begin{equation}\label{eq:cara2}
	\begin{cases}
		\dot{u}_{\xi}(s)=L(s,\xi(s),\dot{\xi}(s),u_{\xi}(s)),\quad s\in[t_1,t_2],&\\
		u_{\xi}(t_1)=u.&
	\end{cases}
\end{equation}
We call the function $h_L(t_1,t_2,x,y,u)$ the \emph{(negative type) fundamental solution} for the Hamilton-Jacobi equation $D_tu(t,x)+H(t,x,D_xu(t,x),u(t,x))=0$.

We denote by $u(t,x)$, $(t,x)\in(0,+\infty)\times\R^n$, the value function of Bolza problem \eqref{eq:cov}. As shown in \cite{CCJWY2020}, the value function can represented as
\begin{align*}
	u(t,x)=&\,\inf_{y\in\R^n}\{u_0(y)+h_L(0,t,y,x,u_0(y))\}\\
	=&\,\inf_{\xi\in\mathcal{A}_{t,x}}\left\{u_0(\xi(0))+\int^t_0L(s,\xi(s),\dot{\xi}(s),u_{\xi}(s))\ ds\right\}
\end{align*}
where $u_{\xi}$ is determined by \eqref{eq:cara1}. Similar to the classical case, $u(t,x)$ is a viscosity solution of \eqref{eq:HJe}.

Now, we list some fundamental results on \eqref{eq:cov}, especially the regularity aspects. We collect all the proofs of these propositions in Section \ref{sec:App_A}. For any $(t,x)\in(0,+\infty)\times\R^n$, set
\begin{align*}
	\mathcal{Z}_{t,x}=\{z\in\R^n: u(t,x)=u_0(z)+h_L(0,t,z,x,u_0(z))\}.
\end{align*}

\begin{Pro}\label{pro:Herglotz_Lie}
For any $(t,x)\in(0,+\infty)\times\R^n$ we have $\mathcal{Z}_{t,x}\neq\varnothing$. If $y_{t,x}\in\mathcal{Z}_{t,x}$ then there exists $\xi\in\Gamma^{0,t}_{y_{t,x},x}$ such that $\xi$ is a minimizer of \eqref{eq:cov}. Moreover, we have that
\begin{enumerate}[\rm (1)]
	\item $\xi$ is of class $C^{R+1}$ and it satisfies the Herglotz' equation
	\begin{equation}
		\begin{split}
			&\,\frac d{ds}L_v(s,\xi(s),\dot{\xi}(s),u_{\xi}(s))\\
			=&\,L_x(s,\xi(s),\dot{\xi}(s),u_{\xi}(s))+L_u(s,\xi(s),\dot{\xi}(s),u_{\xi}(s))L_v(s,\xi(s),\dot{\xi}(s),u_{\xi}(s))
		\end{split}
	\end{equation}
	on $[0,t]$ with $u_\xi$ satisfies Carath\'eodory equation \eqref{eq:cara1}.
	\item Set $p(s)=L_v(s,\xi(s),\dot{\xi}(s),u_{\xi}(s))$. Then the arc $(\xi,p,u_{\xi})$ satisfies the following Lie equation
	\begin{equation}\label{eq:Lie}
  	\begin{cases}
  		\dot{\xi}=H_p(s,\xi,p,u_{\xi}),\\
  		\dot{p}=-H_x(s,\xi,p,u_{\xi})-H_u(s,\xi,p,u_{\xi})p,\qquad s\in[0,t],\\
  		\dot{u}_{\xi}=p\cdot\dot{\xi}-H(s,\xi,p,u_{\xi}).
  	\end{cases}
  	\end{equation}
  	\item There exists $C(t,x)>0$ such that
  	\begin{align*}
  		\max_{s\in[0,t]}\{|\xi(s)|,|\dot{\xi}(s)|,|p(s)|,|u_{\xi}(s)|\}\leqslant C(t,x).
  	\end{align*}
\end{enumerate}
\end{Pro}

\begin{Pro}[dynamic programming principle]\label{pro:dyn_prog}
Let $(t,x)\in(0,+\infty)\times\R^n$ and $\xi\in\mathcal{A}_{t,x}$. Then, for all $0\leqslant t'\leqslant t$,
\begin{equation}\label{eq:dpp}
	u(t,\xi(t))\leqslant u(t',\xi(t'))+\int^{t}_{t'}L(s,\xi,\dot{\xi},u_{\xi})\ ds
\end{equation}
where $u_{\xi}$ satisfies \eqref{eq:cara2} on $[t',t]$ with $u=u(t',\xi(t'))$. The equality holds in \eqref{eq:dpp} if and only if $\xi$ is a minimizer of \eqref{eq:cov}.
\end{Pro}

\begin{Pro}\label{pro:D^*}
\hfill
\begin{enumerate}[\rm (1)]
	\item $u(t,x)$ is a  solution of \eqref{eq:HJe} in the sense of viscosity;
	\item $u$ is locally Lipschitz and locally semiconcave on $(0,+\infty)\times\R^n$;
	\item the following relation holds:
	\begin{align*}
		\mbox{\rm Ext}\,(D^+u(t,x))=D^*u(t,x)=\{(q,p)\in D^+u(t,x): q+H(t,x,p,u(t,x))=0\};
	\end{align*}
	\item if $\xi$ is a minimizer of \eqref{eq:cov} then $u$ is differentiable at $(s,\xi(s))$ for all $0<s<t$.
\end{enumerate}
\end{Pro}

\begin{Pro}\label{pro:sensitive}
Let $(t,x)\in(0,+\infty)\times\R^n$ and $\xi$ be a minimizer of \eqref{eq:cov}. Then
\begin{align*}
	p(t)=&\,L_v(t,\xi(t),\dot{\xi}(t),u_{\xi}(t))\in \nabla^+u(t,x),\\
	p(s)=&\,L_v(s,\xi(s),\dot{\xi}(s),u_{\xi}(s))=\nabla u(s,\xi(s)),\quad \forall s\in(0,t),\\
	p(0)=&\,L_v(0,\xi(0),\dot{\xi}(0),u_{\xi}(0))\in Du_0(\xi(0)).
\end{align*}
\end{Pro}

\begin{Pro}\label{pro:contact_minimizer}
If the triple $(X,P,U)$ of curves satisfies Lie equation \eqref{eq:Lie} on $[0,t]$, $U(0)=u_0(X(0))$ and $U(s)=u(s,X(s))$ for $s\in(0,t]$, then $X$ is a minimizer of \eqref{eq:cov} with $x=X(t)$.
\end{Pro}

\begin{Pro}\label{pro:1-1}
For any $(t,x)\in(0,+\infty)\times\R^n$, we have the following one-to-one correspondences
\begin{align*}
	D^*u(t,x)\quad\Longleftrightarrow\quad \text{minimizers of \eqref{eq:cov}}\quad\Longleftrightarrow\quad \mathcal{Z}_{t,x}.
\end{align*}
More precisely, $(q,p)\in D^*u(t,x)$ if and only if there exists a minimizer $\xi$ of \eqref{eq:cov} such that $p=L_v(t,\xi(t),\dot{\xi}(t),u_{\xi}(t))$ with $u_{\xi}$ determined by \eqref{eq:cara1}. The second one-to-one correspondence is that $\xi$ is a minimizer of \eqref{eq:cov} if and only if $\xi(0)\in\mathcal{Z}_{t,x}$.
\end{Pro}

\begin{Pro}\label{pro:approx}
\hfill
\begin{enumerate}[\rm (1)]
	\item Let $(t_k,x_k),(t,x)\in(0,+\infty)\times\R^n$ ($k\in\N$) and $(t_k,x_k)\to(t,x)$ as $k\to\infty$. If $\xi_k$ is a minimizer of \eqref{eq:cov} with respect to $(t_k,x_k)$, then there exists a subsequence $\{k_i\}$ such that $\xi_{k_i}$ converges to some minimizer of \eqref{eq:cov} as $i\to\infty$ under the $C^{R+1}$-topology.
	\item Given $(t,x)\in(0,+\infty)\times\R^n$. Then, for any $\varepsilon>0$ there exists $\delta>0$ such that, if $|(t',x')-(t,x)|<\delta$, then, for any minimizer $\eta$ of \eqref{eq:cov} with respect to $(t',x')$, there exists a minimizer $\xi$ of \eqref{eq:cov} such that
	\begin{align*}
		\|\xi-\eta\|_{C^{R+1}([0,\min\{t,t'\})]}<\varepsilon.
	\end{align*}
\end{enumerate}
\end{Pro}

\section{Irregular and conjugate points}
\label{sec:irr_conj}

In this section, we will give the definition of the conjugate loci and analyze the structure of the cut locus for problem \eqref{eq:cov}.

\subsection{Irregular and conjugate points}

Consider the following Lie equation or characteristic system
\begin{equation}\label{eq:Lie2}
	\begin{cases}
  		\dot{X}=H_p(t,X,P,U),\\
  		\dot{P}=-H_x(t,X,P,U)-H_u(t,X,P,U)P,\qquad t\geqslant0,\\
  		\dot{U}=P\cdot\dot{X}-H(t,X,P,U),
  	\end{cases}
\end{equation}
with respect to the initial condition
\begin{align*}
	X(0;z)=z,\quad P(0;z)=Du_0(z),\quad U(0;z)=u_0(z),\quad z\in\R^n.
\end{align*}

By differentiating \eqref{eq:Lie2} with respect to $z$ we obtain that the triple $(X_z,P_z,U_z)$ satisfies the (linear) variational equation
\begin{equation}\label{eq:variational_eq}
	\begin{cases}
		\dot{X}_z=H_{px}X_z+H_{pp}P_z+H_{pu}U_z,\\
		\dot{P}_z=-H_{xx}X_z-H_{xp}P_z-H_{x,u}U_z-H_uP_z-(PH_{ux})X_z-(PH_{up})P_z-(PH_{uu})U_z,\\
		\dot{U}_z=P^TH_{px}X_z+P^TH_{pp}P_z+P^TH_{pu}U_z-H_xX_z-H_pP_z-H_uU_z,
	\end{cases}
\end{equation}
with initial condition
\begin{align*}
	X_z(0;z)=I,\quad P_z(0;z)=D^2u_0(z),\quad U_z(0;z)=\nabla u_0(z).
\end{align*}
Now, for any $z\in\R^n$ and $\theta\in\R^n\setminus\{0\}$, $X_z(0;z)\theta=\theta\not=0$. Notice $(X_z(t;z)\theta,P_z(t;z)\theta,U_z(t;z)\theta)$ satisfies an linear ODE from \eqref{eq:variational_eq}. Thus,
\begin{align*}
	(X_z(t;z)\theta,P_z(t;z)\theta,U_z(t;z)\theta)\not=0,\quad\forall t\geqslant0.
\end{align*}
To simplify equation \eqref{eq:variational_eq}, we note that
\begin{equation}\label{eq:Lie_U1}
	\begin{cases}
		\dot{U}=L(t,X,H_p(t,X,P,U),U)=L(t,X,\dot{X},U),\quad t\geqslant0,\\
		U(0,z)=u_0(z),\quad z\in\R^n.
	\end{cases}
\end{equation}
By differentiating \eqref{eq:Lie_U1} with respect to $z$ we have that
\begin{align*}
	\begin{cases}
		\dot{U}_z=L_x^TX_z+L_v^T\dot{X}_z+L_uU_z,\\
		U_z(0;z)=\nabla u_0(z).
	\end{cases}
\end{align*}
One can solve the equation above to obtain
\begin{align*}
	U_z(t;z)=e^{\int^t_0L_udr}\nabla u_0(z)+\int^t_0e^{\int^t_sL_u dr}(L_x^TX_z+L_v^T\dot{X}_z)ds.
\end{align*}
Recalling that from Herglotz equation we have that
\begin{align*}
	\frac{d}{ds}\left\{e^{\int^t_sL_u dr}L_v^TX_z\right\}=e^{\int^t_sL_u dr}(L_x^TX_z+L_v^T\dot{X}_z).
\end{align*}
Therefore,
\begin{equation}\label{eq:solve_u}
	\begin{split}
		U_z(t;z)=&\,e^{\int^t_0L_udr}\nabla u_0(z)+e^{\int^t_sL_u dr}L_v^TX_z\vert^t_0\\
	=&\,e^{\int^t_0L_udr}\nabla u_0(z)+L_v^T(t,X(t;z),\dot{X}(t;z),U(t;z))X_z(t;z)\\
	&\quad -e^{\int^t_0L_u dr}L_v^T(0,X(0;z),\dot{X}(0;z),U(0;z))X_z(0;z)\\
	=&\,e^{\int^t_0L_udr}\nabla u_0(z)+P^T(t;z)X_z(t;z)-e^{\int^t_0L_udr}\nabla u_0(z)\\
	=&\,P^T(t;z)X_z(t;z).
	\end{split}
\end{equation}
So, if $\det X_z(t;z)=0$, then there exists $\theta\in\R^n\setminus\{0\}$ such that $X_z(t;z)\theta=0$. It follows $U_z(t;z)\theta=P^T(t;z)X_z(t;z)\theta=0$. This leads to the fact that $P_z(t;z)\theta\not=0$.

Now we can introduce some notions from calculus of variation in the contact case.

\begin{defn}[irregular point and conjugate point]\hfill
\begin{enumerate}[\rm (1)]
	\item $(t,x)\in(0,+\infty)\times\R^n$ is called \emph{regular} if problem \eqref{eq:cov} admits a unique solution. Otherwise, we say $(t,x)$ is \emph{irregular}.
	\item $(t,x)$ is called \emph{a conjugate point} if there exists $z\in\R^n$ such that $X(t;z)=x$, $X$ is minimizer of \eqref{eq:cov} and $\det X_z(t;z)=0$.
	\item The set of all irregular points and conjugate points for problem \eqref{eq:cov} are denoted by $\Sigma$ and $\Gamma$ respectively.
\end{enumerate}
\end{defn}

\begin{Rem}
From Proposition \ref{pro:1-1}, the set $\Sigma$ is exactly $\SING$, the set of points at which $u$ is non-differentiable.
\end{Rem}

\subsection{Structural analysis of cut locus}

\begin{Pro}
For any $(t,x)\in(0,+\infty)\times\R^n$, let $\widetilde{\mathcal{Z}}(t,x)=\{z\in\R^n: X(t;z)=x)\}$. Then
\begin{align*}
	u(t,x)=&\,\min\{U(t;z): z\in\widetilde{\mathcal{Z}}(t,x)\},\\
	\mathcal{Z}(t,x)=&\,\{z\in\widetilde{\mathcal{Z}}(t,x): U(t;z)=u(t;z)\}.
\end{align*}
\end{Pro}

\begin{proof}
Proposition \ref{pro:Herglotz_Lie} and \ref{pro:sensitive} shows any minimizer $\xi:[0,t]\to\R^{n}$ for $u(t,x)$ is a solution of \eqref{eq:Lie} and there hold $u_{\xi}(0)=u_{0}(\xi(0))$, $p(0)=Du_{0}(\xi(0))$. Thus, we know that $\xi(\cdot)=X(\cdot,\xi(0))$ and $u_{\xi}(\cdot)=U(\cdot,\xi(0))$, which imply $\xi(0)\in\widetilde{\mathcal{Z}}(t,x)$ and $u(t,x)=u_{\xi}(t)=U(t,\xi(0))$. On the other hand, for any $z\in\widetilde{\mathcal{Z}}(t,x)$, we have
\begin{align*}
U(t,z)=u_{0}(z)+\int_{0}^{t}L(s,X(s,z),\dot{X}(s,z),U(s,z))\geqslant u(t,x).
\end{align*}
This leads to our conclusion.
\end{proof}

Now, let $(t,x)\in(0,+\infty)\times\R^n$ and $\xi$ is a minimizer of \eqref{eq:cov}. Given a Lipschitz curve $\alpha:[0,t]\to\R^n$ with $\alpha(t)=0$, we note that $\xi^{\varepsilon}=\xi+\varepsilon\alpha\in\mathcal{A}_{t,x}$ for any $\varepsilon\in[-1,1]$. Then the Carath\'eodory equation with respect to $\xi^{\varepsilon}$ is
\begin{align*}
	\begin{cases}
		\dot{u}_{\xi^{\varepsilon}}=L(s,\xi^{\varepsilon},\dot{\xi}^{\varepsilon},u_{\xi^{\varepsilon}}),\quad a.e., s\in[0,t],\\
		u_{\xi^{\varepsilon}}(0)=u_0(\xi^{\varepsilon}(0)).
	\end{cases}
\end{align*}
By differentiating with respect to $\varepsilon$ we have
\begin{equation}\label{eq:diff1}
	\begin{cases}
		\frac{\partial}{\partial\varepsilon}\dot{u}_{\xi^{\varepsilon}}=L^T_x(s,\xi^{\varepsilon},\dot{\xi}^{\varepsilon},u_{\xi^{\varepsilon}})\cdot\alpha+L^T_v(s,\xi^{\varepsilon},\dot{\xi}^{\varepsilon},u_{\xi^{\varepsilon}})\cdot\dot{\alpha}+L_u(s,\xi^{\varepsilon},\dot{\xi}^{\varepsilon},u_{\xi^{\varepsilon}})\frac{\partial}{\partial\varepsilon}u_{\xi^{\varepsilon}},\\
	\frac{\partial}{\partial\varepsilon}u_{\xi^{\varepsilon}}(0)=\nabla u_0(\xi^{\varepsilon}(0))\cdot\alpha(0).
	\end{cases}
\end{equation}
Similar to  \eqref{eq:solve_u} we can solve the equation above to obtain
\begin{align*}
	\frac{\partial}{\partial\varepsilon}\Big\vert_{\varepsilon=0}u_{\xi^{\varepsilon}}(s)=L^T_v(s,\xi(s),\dot{\xi}(s),u_{\xi}(s))\cdot\alpha(s),\quad s\in[0,t].
\end{align*}
In particular,
\begin{align*}
	\frac{\partial}{\partial\varepsilon}\Big\vert_{\varepsilon=0}u_{\xi^{\varepsilon}}(t)=0.
\end{align*}
Differentiating \eqref{eq:diff1} again at $\varepsilon=0$ leads to
\begin{align*}
	\begin{cases}
		\frac{\partial^2}{\partial\varepsilon^2}\Big\vert_{\varepsilon=0}\dot{u}_{\xi^{\varepsilon}}=\alpha^TL_{xx}\alpha+\alpha^TL_{xv}\dot{\alpha}+\alpha^TL_{xu}\frac{\partial}{\partial\varepsilon}\Big\vert_{\varepsilon=0}u_{\xi^{\varepsilon}}\\
		\hskip 2.5 cm +\dot{\alpha}^TL_{vx}\alpha+\dot{\alpha}^TL_{vv}\dot{\alpha}+\dot{\alpha}^TL_{vu}\frac{\partial}{\partial\varepsilon}\Big\vert_{\varepsilon=0}u_{\xi^{\varepsilon}}\\
		\hskip 2.5 cm +L_{ux}\alpha\cdot\frac{\partial}{\partial\varepsilon}\Big\vert_{\varepsilon=0}u_{\xi^{\varepsilon}}+L_{uv}\dot{\alpha}^T\cdot\frac{\partial}{\partial\varepsilon}\Big\vert_{\varepsilon=0}u_{\xi^{\varepsilon}}+L_{uu}\left(\frac{\partial}{\partial\varepsilon}\Big\vert_{\varepsilon=0}u_{\xi^{\varepsilon}}\right)^2+L_u\cdot\frac{\partial^2}{\partial\varepsilon^2}\Big\vert_{\varepsilon=0}u_{\xi^{\varepsilon}},\\
		\hskip 1.8 cm =\alpha^T(L_{xx}+2L_{xu}L^T_v+L_vL_{uu}L^T_v)\alpha+2\alpha^T(L_{xv}+L_vL_{uv})\dot{\alpha}+\dot{\alpha}^TL_{vv}\dot{\alpha}+L_u\frac{\partial^2}{\partial\varepsilon^2}\Big\vert_{\varepsilon=0}u_{\xi^{\varepsilon}}\\
		\frac{\partial^2}{\partial\varepsilon^2}\Big\vert_{\varepsilon=0}u_{\xi^{\varepsilon}}(0)=\alpha^T(0)D^2u_0(\xi(0))\alpha(0).
	\end{cases}
\end{align*}
Solving the ODE above we obtain
\begin{align*}
	&\,\frac{\partial^2}{\partial\varepsilon^2}\Big\vert_{\varepsilon=0}u_{\xi^{\varepsilon}}(s)\\
	=&\,e^{\int^s_0L_udr}\alpha^T(0)D^2u_0(\xi(0))\alpha(0)\\
	&\quad +\int^s_0e^{\int^s_{\tau}L_udr}\{\alpha^T(L_{xx}+2L_{xu}L^T_v+L_vL_{uu}L^T_v)\alpha+2\alpha^T(L_{xv}+L_vL_{uv})\dot{\alpha}+\dot{\alpha}^TL_{vv}\dot{\alpha}\}\ d\tau
\end{align*}
for all $s\in[0,t]$. Define
\begin{align*}
	J^*(\alpha)=\frac{\partial^2}{\partial\varepsilon^2}\Big\vert_{\varepsilon=0}u_{\xi^{\varepsilon}}(t).
\end{align*}
Obviously $J^*(\alpha)\geqslant0$ since $\xi$ is a minimizer of \eqref{eq:cov}.

Observe that, for a new Lagrangian
\begin{align*}
	&\,\tilde{L}(s,\alpha,\dot{\alpha},u)\\
	=&\,\alpha^T(L_{xx}+2L_{xu}L^T_v+L_vL_{uu}L^T_v)\alpha+2\alpha^T(L_{xv}+L_vL_{uv})\dot{\alpha}+\dot{\alpha}^TL_{vv}\dot{\alpha}+L_u\cdot u.
\end{align*}
$\tilde{L}$ satisfies all of conditions in (L1)-(L4) except for the fact that $\tilde{L}$ is only continuous in the time variable. We can not directly apply the results from \cite{CCJWY2020} to guarantee the $C^1$-regularity of the minizers of $J^*$.

\begin{Lem}\label{mini c1}
Suppose $\alpha:[0,t]\to \R^{n},\alpha(t)=0$ is a Lipschitz curve. If $J^{*}(\alpha)=\inf\{J^{*}(\gamma)|\gamma:[0,t]\to \R^{n},\gamma(t)=0 \text{ is a Lipschitz curve}\}$, we have $\alpha\in C^{1}([0,t],\R^{n})$.
\end{Lem}
\begin{proof}
Notice that $\tilde{L}$ satisfies conditions (L1)-(L3) and (L4') in \cite{CCJWY2020} and $\alpha$ is a minimizer for $J^*$. Therefore, the results in \cite{CCJWY2020} implies that $\alpha$ is of class $C^1$.
\end{proof}

\begin{Pro}\label{mini int regu non-conj}
Suppose $(t,x)\in (0,\infty)\times \R^{n},\ z_{0}\in\mathcal{Z}(t,x)$. Then $(s,X(s,z_{0}))\notin\Sigma\cup\Gamma$ for all $s\in(0,t)$.
\end{Pro}
\begin{proof}
Fix any $\bar{s}\in (0,t)$. Proposition \ref{pro:D^*} (4) implies $(\bar{s},X(\bar{s},z_{0}))\notin \Sigma$ and we prove $(\bar{s},X(\bar{s},z_{0}))\notin \Gamma$ by contradiction.
If $(\bar{s},X(\bar{s},z_{0}))\in \Gamma$, then there exists $\theta\in \R^{n}\setminus \{0\}$ such that $X_{z}(\bar{s},z_{0})\theta=0$. By \eqref{eq:solve_u} there holds $U_{z}(\bar{s},z_{0})\theta=0$ and $P_{z}(\bar{s},z_{0})\theta\neq 0$. Let
\begin{align*}
\alpha(s)=
\begin{cases}
X_{z}(s,z_{0})\theta,& s\in[0,\bar{s}],\\
0,& s\in[\bar{s},t].
  	\end{cases}
\end{align*}
Then $\alpha$ is a Lipschitz curve, $\alpha(0)=\theta,\alpha(t)=0,\dot{\alpha}^{+}(\bar{s})=0$ and
\begin{align*}
\dot{\alpha}^{-}(\bar{s})=\dot{X}_{z}(\bar{s},z_{0})\theta=H_{px}X_{z}(\bar{s},z_{0})\theta+H_{pp}P_{z}(\bar{s},z_{0})\theta+H_{pu}U_{z}(\bar{s},z_{0})\theta
=H_{pp}P_{z}(\bar{s},z_{0})\theta \neq 0.
\end{align*}
So we have $\alpha\notin C^{1}([0,t],\R^{n})$.
On the other hand, by previous computation we obtain
\begin{align*}
&J^{*}(\alpha)=e^{\int_{0}^{t}L_{u} dr}\theta^{T}D^{2}u_{0}(z_{0})\theta
+e^{\int_{0}^{t}L_{u} dr}\int_{0}^{\bar{s}}e^{-\int_{0}^{s}L_{u} dr}[\theta^{T}X_{z}^{T}(L_{xx}+2L_{xu}L_{v}^{T}+L_{v}L{uu}L_{v}^{T})X_{z}\theta\\
&+2\theta^{T}X_{z}^{T}(L_{xv}+L_{v}L_{uv})\dot{X}_{z}\theta+\theta^{T}\dot{X}_{z}^{T}L_{vv}\dot{X}_{z}\theta] ds.
\end{align*}
Notice
\begin{align*}
\begin{cases}
\frac{\partial}{\partial z}L_{x}=L_{xx}X_{z}+L_{xv}\dot{X}_{z}+L_{xu}L_{v}^{T}X_{z}\\
\frac{\partial}{\partial z}L_{v}=L_{vx}X_{z}+L_{vv}\dot{X}_{z}+L_{vu}L_{v}^{T}X_{z},\ t\in[0,\infty),z\in\R^{n}\\
\frac{\partial}{\partial z}L_{u}=L_{ux}X_{z}+L_{uv}\dot{X}_{z}+L_{uu}L_{v}^{T}X_{z},
\end{cases}
\end{align*}
It follows that
\begin{equation*}
\begin{split}
J^{*}(\alpha)=&e^{\int_{0}^{t}L_{u} dr}\theta^{T}[D^{2}u_{0}(z_{0})\\
              &+\int_{0}^{\bar{s}}e^{-\int_{0}^{s}L_{u} dr}(X_{z}^{T}\frac{\partial}{\partial z}\Big\vert_{z=z_{0}}L_{x}+\dot{X}_{z}^{T}\frac{\partial}{\partial z}\Big\vert_{z=z_{0}}L_{v}+X_{z}^{T}L_{v}\frac{\partial}{\partial z}\Big\vert_{z=z_{0}}L_{u})ds]\theta\\
             =&e^{\int_{0}^{t}L_{u} dr}\theta^{T}[D^{2}u_{0}(z_{0})+\int_{0}^{\bar{s}}e^{-\int_{0}^{s}L_{u} dr}(\frac{d}{ds}(X_{z}^{T}\frac{\partial}{\partial z}\Big\vert_{z=z_{0}}L_{v})-L_{u}X_{z}^{T}\frac{\partial}{\partial z}\Big\vert_{z=z_{0}}L_{v})ds]\theta\\
             =&e^{\int_{0}^{t}L_{u} dr}\theta^{T}[D^{2}u_{0}(z_{0})+(e^{-\int_{0}^{s}L_{u} dr}X_{z}^{T}\frac{\partial}{\partial z}\Big\vert_{z=z_{0}}L_{v})\Big\vert_{s=0}^{s=\bar{s}}]\theta\\
             =&e^{\int_{0}^{t}L_{u} dr}\theta^{T}[D^{2}u_{0}(z_{0})+0-X_{z}^{T}(0,z_{0})\frac{\partial}{\partial z}\Big\vert_{z=z_{0}}L_{v}(0,X(0,z),\dot{X}(0,z),U(0,z))]\theta\\
             =&e^{\int_{0}^{t}L_{u} dr}\theta^{T}[D^{2}u_{0}(z_{0})-\frac{\partial}{\partial z}\Big\vert_{z=z_{0}}Du_{0}(z)]\theta\\
             =&0.
\end{split}
\end{equation*}
Recalling the fact that for any Lipschitz curve $\gamma:[0,t]\to\R^{n}$ with $\gamma(t)=0$, we have $J^{*}(\gamma)\geqslant 0$. Then by Lemma \ref{mini c1}, we conclude that $\alpha\in C^{1}([0,t],\R^{n})$. This leads to a contradiction. Therefore, $(\bar{s},X(\bar{s},z_{0}))\notin \Gamma$.
\end{proof}

\begin{Pro}\label{t small regu non-conj}
For any $\bar{z}\in \R^{n}$, there exists $r_{\bar{z}}>0$ and $t_{\bar{z}}>0$ such that $([0,t_{\bar{z}}]\times B_{r_{\bar{z}}}(\bar{z}))\cap(\Sigma\cup\Gamma)=\varnothing$.
\end{Pro}

\begin{proof}
Let $\phi:[0,\infty)\times\R^{n}\to[0,\infty)\times\R^{n},\ (t,z)\to(t,X(t,z))$. Then $\phi$ is of class $C^{R}$ and $X_{z}(0,z)=Id,\ \forall z\in\R^{n}$. By the inverse function theorem, there exists $t'_{\bar{z}}>0$ and $r_{\bar{z}}>0$ such that $\phi$ is a $C^{R}$ diffeomorphism on $[0,t'_{\bar{z}}]\times B_{2r_{\bar{z}}}(\bar{z})$. Let $t_{\bar{z}}=\min\{t'_{\bar{z}},\frac{r_{\bar{z}}}{C(t'_{\bar{z}},|\bar{z}|+r_{\bar{z}})}\}$, where $C$ is defined in Proposition \ref{pro:Herglotz_Lie} (3). Now fix any $(t,x)\in [0,t_{\bar{z}}]\times B_{r_{\bar{z}}}(\bar{z})$. If $z\in \mathcal{Z}(t,x)$, we have $|z-\bar{z}| \leqslant |z-x|+|x-\bar{z}| \leqslant t_{\bar{z}}C(t,_{\bar{z}},|\bar{z}|+r_{\bar{z}})+r_{\bar{z}} \leqslant r_{\bar{z}}+r_{\bar{z}} = 2r_{\bar{z}}$, that is $\mathcal{Z}(t,x)\subset B_{2r_{\bar{z}}}(\bar{z})$. This implies $\mathcal{Z}(t,x)=z(t,x)$ is a singleton and $X_{z}(t,z(t,x))$ is invertible. Therefore $(t,x)\notin \Sigma\cup\Gamma$. In conclusion, $([0,t_{\bar{z}}]\times B_{r_{\bar{z}}}(\bar{z}))\cap(\Sigma\cup\Gamma)=\varnothing$.
\end{proof}

\begin{Pro}\label{regu non-conj pro}
Suppose $(\bar{t},\bar{x})\in (0,\infty)\times \R^{n},\ (\bar{t},\bar{x})\notin \Sigma\cup\Gamma$ and $\mathcal{Z}(\bar{t},\bar{x})=\bar{z}$. Then there exists a neighborhood $W$ of $(\bar{t},\bar{z})$ and a neighborhood $N$ of $(\bar{t},\bar{x})$ such that $\phi:W\to N,\ (t,z)\to(t,X(t,z))$ is a $C^{R}$ diffeomorphism and we have:
\begin{enumerate}[\rm (1)]
	\item For any $(t,x)\in N$, $\mathcal{Z}(t,x)=\Pi_{z}(\phi^{-1}(t,x))$ is a singleton, i.e., $N\cap(\Sigma\cup\Gamma)=\varnothing$.
	\item For all $(t,x)\in N$,
	\begin{align*}
		u(t,x)=&\,U(\phi^{-1}(t,x)),\\
		\nabla u(t,x)=&\,P(\phi^{-1}(t,x)),\\
		u_{t}(t,x)=&\,-H(t,x,P(\phi^{-1}(t,x)),U(\phi^{-1}(t,x))).
	\end{align*}
	Moreover, $u\in C^{R+1}(N,\R)$.
\end{enumerate}
\end{Pro}

\begin{proof}
$(\bar{t},\bar{x})\notin \Sigma\cup\Gamma$ implies $X_{z}(\bar{t},\bar{z})$ is non-degenerate. By the inverse function theorem, we know that there exists a neighborhood $W$ of $(\bar{t},\bar{z})$ and a neighborhood $N$ of $(\bar{t},\bar{x})$, such that $\phi:W\to N,\ (t,z)\to(t,X(t,z))$ is a $C^{R}$ diffeomorphism.
By Proposition \ref{pro:approx} (2), we can choose $N$ small enough such that for any $(t,x)\in N,\ \mathcal{Z}(t,x)=\Pi_{z}(\phi^{-1}(t,x))=z(t,x)$ is a singleton. That is, $N\cap(\Sigma\cup\Gamma)=\varnothing$. This completes the proof of (1).

It follows from (1) that $u(t,x)=U(\phi^{-1}(t,x))$ and $\nabla u(t,x)=P(\phi^{-1}(t,x)),\ \forall (t,x)\in N$. Then by Proposition \ref{pro:D^*} (1), we have $u_{t}(t,x)=-H(t,x,P(\phi^{-1}(t,x)),U(\phi^{-1}(t,x)))$. Therefore, $\nabla u(t,x)$ and $u_{t}(t,x)$ are of class $C^{R}$, which implies $u\in C^{R+1}(N,\R)$.
\end{proof}

\begin{Cor}\label{regu non-conj cor}
$\Sigma\cup\Gamma$ is a closed set, and $u$ is of class $C^{R+1}$ on $(\Sigma\cup\Gamma)^{c}$.
\end{Cor}

\begin{Pro}\label{non-conj Z finite}
Suppose $(t,x)\in (0,\infty)\times\R^{n},\ (t,x)\notin\Gamma$. Then $\mathcal{Z}(t,x)$ and $D^{*}u(t,x)$ are finite sets.
\end{Pro}

\begin{proof}
For any $\bar{z}\in \mathcal{Z}(t,x)$, $X_{z}(t,\bar{z})$ is non-degenerate. Then there exists a neighborhood $V$ of $\bar{z}$ such that $X(t,z)\neq x,\ \forall z\in V$. This implies $\mathcal{Z}(t,x)$ consists of isolated points. On the other hand, by Proposition \ref{pro:Herglotz_Lie} (3), $\mathcal{Z}(t,x)$ is bounded. Therefore, $\mathcal{Z}(t,x)$ is a finite set. Combing this with Proposition \ref{pro:1-1}, we conclude that $D^{*}u(t,x)$ is also a finite set.
\end{proof}

\begin{Pro}\label{irregu non-conj pro}
Suppose $(\bar{t},\bar{x})\in (0,\infty)\times \R^{n}$, $(\bar{t},\bar{x})\in\Sigma\setminus \Gamma$. By Proposition \ref{non-conj Z finite}, we let $\mathcal{Z}(t,x)=\{z_{1},z_{2},\cdots,z_{k}\},k\geqslant 2$. Then there exist neighborhoods $W_{i}$ of $(\bar{t},z_{i})$ and a neighborhood $N$ of $(\bar{t},\bar{x})$ such that $\phi_{i}:W_{i}\to N$, $(t,z)\to(t,X(t,z))$, $i=1,2,\ldots,k$ are all $C^{R}$ diffeomorphisms and we have:
\begin{enumerate}[\rm (1)]
	\item $W_{i}\cap W_{j}=\varnothing$ and $P(W_{i})\cap P(W_{j})=\varnothing$ for all $i,j\in\{1,2,\ldots,k\}$, $i\neq j$.
	\item Let $v_{i}(t,x)=U(\phi_{i}^{-1}(t,x))$, $(t,x)\in N$, $i=1,2,\ldots,k$. Then $v_{i}\in C^{R+1}(N,\R)$ and
	\begin{align*}
		\nabla v_{i}(t,x)=&\,P(\phi_{i}^{-1}(t,x)),\\
		\frac{\partial}{\partial t}v_{i}(t,x)=&\,-H(t,x,P(\phi_{i}^{-1}(t,x)),U(\phi_{i}^{-1}(t,x)))
	\end{align*}
	for all $(t,x)\in N$, $i=1,2,\ldots,k$.
    \item We have that $u(t,x)=\min_{1\leqslant i\leqslant k}\{v_{i}(t,x)\}$ and
    \begin{align*}
    	\mathcal{Z}(t,x)=&\,\{\Pi_{z}(\phi_{i}^{-1}(t,x)): v_{i}(t,x)=u(t,x), 1\leqslant i\leqslant k\}\\
    	D^{*}u(t,x)=&\,\{(\frac{\partial}{\partial t}v_{i}(t,x),\nabla v_{i}(t,x)): v_{i}(t,x)=u(t,x), 1\leqslant i\leqslant k\}
    \end{align*}
    for all $(t,x)\in N$. Moreover, $N\cap\Gamma=\varnothing$.
    \item $N\cap\Sigma$ is contained in a finite union of $n$-dimensional hypersurfaces  of class $C^{R+1}$.
\end{enumerate}
\end{Pro}

\begin{proof}
$(\bar{t},\bar{x})\notin(\Sigma\cup\Gamma)$ implies that $X_{z}(\bar{t},z_{i}),\ i=1,2,\cdots,k$ are non-degenerate. By the inverse function theorem, we know that there exist neighborhoods $W_{i}$ of $(\bar{t},z_{i})$ and a neighborhood $N$ of $(\bar{t},\bar{x})$ such that $\phi_{i}:W_{i}\to N,\ (t,z)\to(t,X(t,z)),\ i=1,2,\cdots,k$ are $C^{R}$ diffeomorphisms.

By Proposition \ref{pro:1-1}, $(-H(\bar{t},\bar{x},P(\bar{t},z_{i}),u(\bar{t},\bar{x})),P(\bar{t},z_{i}))$, $1\leqslant i\leqslant k$, are all different. This implies $P(\bar{t},z_{i})\neq P(\bar{t},z_{j}),\ i\neq j$. Therefore, by taking $N$ sufficiently small, we can make $W_{i}\cap W_{j}=\varnothing$ and $P(W_{i})\cap P(W_{j})=\varnothing$ for all $i,j\in\{1,2,\cdots,k\}$, $i\neq j$. This completes the proof of (1)

In fact, $v_{i}$ can be seen as a solution of \eqref{eq:cov} when $\xi(0)$ is limited in $W_{i}$. In this sense, $v_{i}$ is regular and non-conjugate on $N$, $1\leqslant i\leqslant k$. Then the conclusion in (2) is a direct consequence of Proposition \ref{regu non-conj pro}.

Now we turn to the proof of (3). By Proposition \ref{pro:approx} (2), we can choose $N$ small enough such that $\mathcal{Z}(t,x)\subset\bigcup_{i=1}^{k}W_{i}$ for all $(t,x)\in N$. Then we have $u(t,x)=\min_{1\leqslant i\leqslant k}\{v_{i}(t,x)\}$ and $\mathcal{Z}(t,x)=\{\Pi_{z}(\phi_{i}^{-1}(t,x)): v_{i}(t,x)=u(t,x), 1\leqslant i\leqslant k\}$ for all $(t,x)\in N$, which also implies $N\cap\Gamma=\varnothing$. Combing this with Proposition \ref{pro:1-1} and (2), we obtain $D^{*}u(t,x)=\{(\frac{\partial}{\partial t}v_{i}(t,x),\nabla v_{i}(t,x)): v_{i}(t,x)=u(t,x), 1\leqslant i\leqslant k\}$ for all $(t,x)\in N$.

Finally, we can shrink $N$ again such that for any $i=1,2,\cdots,k$ and $(t,z)\in W_{i}$, $X_{z}(t,z)$ is non-degenerate, which implies $N\cap\Gamma=\varnothing$. When $i\neq j$, we have $v_{i}(\bar{t},\bar{x})=v_{j}(\bar{t},\bar{x})=u(\bar{t},\bar{x})$ and $Dv_{i}(\bar{t},\bar{x})-Dv_{j}(\bar{t},\bar{x})\neq 0$. By the implicit function theorem, we can take $N$ sufficiently small such that $\{(t,x)\in N: v_{i}(t,x)=v_{j}(t,x)\}$, $i\neq j$, are all $n$-dimensional hyper-surfaces  of class $C^{R+1}$. Notice that for any $(t',x')\in N\cap\Sigma$, there exists $i\neq j$ such that $(t',x')\in\{(t,x)\in N: v_{i}(t,x)=v_{j}(t,x)\}$. Therefore, $N\cap\Sigma$ is contained in a finite union of n-dimensional hyper-surfaces  of class $C^{R+1}$. This completes the proof of (4).
\end{proof}

\begin{Cor}\label{cor:conj closed}
$\Gamma$ is a closed set.
\end{Cor}

\begin{proof}
Indeed, $\Gamma^{c}=(\Sigma\cup\Gamma)^{c}\cup(\Sigma\setminus\Gamma)$. Choose any $(t,x)\in\Gamma^{c}$. If $(t,x)\in(\Sigma\cup\Gamma)^{c}$, by Proposition 3.7 (1), there exists a neighborhood $N$ of $(t,x)$ such that $N\subset (\Sigma\cup\Gamma)^{c}\subset \Gamma^{c}$. If $(t,x)\in\Sigma\setminus\Gamma$, Proposition \ref{irregu non-conj pro} (3) implies there exists a neighborhood $N'$ of $(t,x)$ such that $N'\subset\Gamma^{c}$. Therefore, $\Gamma^{c}$ is open, that is, $\Gamma$ is closed.
\end{proof}

\begin{Pro}
Suppose $(t_{0},x_{0})\in\Sigma\cup\Gamma$. Let $M=C(t_{0}+1,|x_{0}|+1)$, $h=\min\{1,M^{-1}\}$, where $C$ is defined in Proposition \ref{pro:Herglotz_Lie} (3). Then for any $\varepsilon\in(0,h]$, there exists $x_{\varepsilon}\in B_{\varepsilon M}(x_{0})$ such that $(t_{0}+\varepsilon,x_{\varepsilon})\in\Sigma$.
\end{Pro}

\begin{proof}
We prove it by contradiction. If there exists $\varepsilon_{0}\in(0,h]$ such that $(t_{0}+\varepsilon_{0},x)\notin\Sigma,\ \forall x\in B_{\varepsilon_{0}M}(x_{0})$, then by semiconcavity of $u$, $Du(t_{0}+\varepsilon_{0},x)$ exists for all $x\in B_{\varepsilon_{0}M}(x_{0})$ and is continuous with $x$ on $B_{\varepsilon_{0}M}(x_{0})$. For $x\in B_{\varepsilon_{0}M}(x_{0})$, let $\bar{X}(t,x),\bar{P}(t,x),\bar{U}(t,x),t\in[0,t_{0}+\varepsilon_{0}]$ be the solution of \eqref{eq:Lie2} with terminal condition
\begin{align*}
\begin{cases}
\bar{X}(t_{0}+\varepsilon_{0},x)=x\\
\bar{P}(t_{0}+\varepsilon_{0},x)=Du(t_{0}+\varepsilon_{0},x)\\
\bar{U}(t_{0}+\varepsilon_{0},x)=u(t_{0}+\varepsilon_{0},x).
\end{cases}
\end{align*}
Then $\bar{X},\bar{P},\bar{U}$ is continuous with $(t,x)$ and for any $x\in B_{\varepsilon_{0}M}(x_{0})$, $\bar{X}(\cdot,x):[0,t_{0}+\varepsilon_{0}]\to\R^{n}$ is the unique minimizer for $u(t_{0}+\varepsilon_{0},x)$.
Let $\Lambda(x)=x_{0}-\bar{X}(t_{0},x)+x,\ x\in B_{\varepsilon_{0}M}(x_{0})$. Then $\Lambda(x)$ is continuous with $x$, and when $|x-x_{0}|\leqslant \varepsilon_{0}M$, we have $|\Lambda(x)-x_{0}|=|x-\bar{x}(t_{0},x)|=|\bar{x}(t_{0}+\varepsilon_{0},x)-\bar{x}(t_{0},x)|\leqslant\varepsilon_{0}M$. By Brouwer's fixed point theorem, there exists $x_{\varepsilon_{0}}$ such that $\Lambda(x_{\varepsilon_{0}})=x_{\varepsilon_{0}}$, that is $\bar{X}(t_{0},x_{\varepsilon_{0}})=x_{0}$. This implies $(t_{0},x_{0})\notin\Sigma\cup\Gamma$ by Proposition \ref{mini int regu non-conj} and leads to a contradiction with $(t_{0},x_{0})\in\Sigma\cup\Gamma$. Therefore, for any $\varepsilon\in(0,h]$, there exists $x_{\varepsilon}\in B_{\varepsilon M}(x_{0})$ such that $(t_{0}+\varepsilon,x_{\varepsilon})\in\Sigma$.
\end{proof}

\begin{Cor}
$\Sigma\cup\Gamma=\bar{\Sigma}$.
\end{Cor}

\begin{Pro}
Suppose $(t_{0},x_{0})\in\Gamma$ and $z_{0}\in\mathcal{Z}(t_{0},x_{0})$ such that $X_{z}(t_{0},x_{0})$ is degenerate. Then
\begin{enumerate}[\rm (1)]
	\item $\lim_{t\to t_{0}^{-}}\|\nabla^{2}u(t,X(t,z_{0}))\|=+\infty$.
    \item If in addition $(t_{0},x_{0})\in\Gamma\setminus \Sigma$ and $\{x_{i}\},\{p_{i}\}$ satisfy $\lim_{i\to \infty}x_{i}=x_{0}$, $\lim_{i\to \infty}\frac{x_{i}-x_{0}}{|x_{i}-x_{0}|}=v\notin\mbox{\rm Im}\,X_{z}(t_{0},z_{0})$ and $(-H(t_{0},x_{i},p_{i},u(t_{0},x_{i})),p_{i})\in D^{*}u(t_{0},x_{i})$. Then we have
    \begin{align*}
    	\lim_{i\to \infty}\left|\frac{p_{i}-p_{0}}{x_{i}-x_{0}}\right|=+\infty.
    \end{align*}
\end{enumerate}
\end{Pro}

\begin{proof}
By Proposition \ref{mini int regu non-conj} and Proposition \ref{regu non-conj pro}, we know that for any $\varepsilon\in(0,t_{0})$, there exists $r_{\varepsilon}>0$ such that:
\begin{align*}
P(t,z)=\nabla u(t,X(t,z)),\ \forall (t,z)\in [0,t_{0}-\varepsilon]\times B_{r_{\varepsilon}}(z_{0}).
\end{align*}
It follows that
\begin{align*}
P_{z}(t,z)=\nabla^{2} u(t,X(t,z))X_{z}(t,z),\ \forall (t,z)\in [0,t_{0}-\varepsilon]\times B_{r_{\varepsilon}}(z_{0}).
\end{align*}
Since $\varepsilon$ is arbitrary, we obtain that
\begin{align*}
P_{z}(t,z_{0})=\nabla^{2} u(t,X(t,z_{0}))X_{z}(t,z_{0}),\ \forall t\in[0,t_{0}).
\end{align*}
Because $X_{z}(t_{0},x_{0})$ is degenerate, there exists $\theta\in\R^{n}\setminus \{0\}$ such that $X_{z}(t_{0},z_{0})\theta=0$, which implies $P_{z}(t_{0},z_{0})\theta\neq0$. And we have
\begin{align*}
P_{z}(t,z_{0})\theta=\nabla^{2} u(t,X(t,z_{0}))X_{z}(t,z_{0})\theta,\ \forall t\in[0,t_{0}).
\end{align*}
Therefore,
\begin{align*}
\lim_{t\to t_{0}^{-}}\|\nabla^{2}u(t,X(t,z_{0}))\| \geqslant \lim_{t\to t_{0}^{-}}\left|\frac{P_{z}(t,z_{0})\theta}{X_{z}(t,z_{0})\theta}\right| =+\infty.
\end{align*}
Now, we turn to the proof of (2). By Proposition \ref{pro:1-1}, there exists $z_{i}\in\mathcal{Z}(t_{0},x_{i})$ such that $P(t_{0},z_{i})=p_{i},\ i=1,2,\cdots$. Proposition \ref{pro:approx} (2) implies $\lim_{i\to \infty}z_{i}=z_{0}$. For any subsequence $\{z_{i'}\}$ of $\{z_{i}\}$, there exists a sub-subsequence $\{z_{i_{k}}\}$ such that $\lim_{k\to\infty}\frac{z_{i_{k}}-z_{0}}{|z_{i_{k}}-z_{0}|}=\nu\in\R^{n}\setminus \{0\}$. Now we have
\begin{align*}
&p_{i_{k}}-p_{0}=P(t_{0},z_{i_{k}})-P(t_{0},z_{0})=P_{z}(t_{0},z_{0})(z_{i_{k}}-z_{0})+o(|z_{i_{k}}-z_{0}|)\\
&x_{i_{k}}-x_{0}=X(t_{0},z_{i_{k}})-X(t_{0},z_{0})=X_{z}(t_{0},z_{0})(z_{i_{k}}-z_{0})+o(|z_{i_{k}}-z_{0}|),\\
\end{align*}
for $k=1,2,\ldots$. It follows that
\begin{align*}
&\frac{p_{i_{k}}-p_{0}}{|z_{i_{k}}-z_{0}|}=P_{z}(t_{0},z_{0})\frac{z_{i_{k}}-z_{0}}{|z_{i_{k}}-z_{0}|}+o(1)\\
&\frac{x_{i_{k}}-x_{0}}{|z_{i_{k}}-z_{0}|}=X_{z}(t_{0},z_{0})\frac{z_{i_{k}}-z_{0}}{|z_{i_{k}}-z_{0}|}+o(1),\quad k=1,2,\ldots.
\end{align*}
Notice that if $X_{z}(t_{0},z_{0})\nu=\theta\neq0$, then $v=\lim_{k\to\infty}\frac{x_{i_{k}}-x_{0}}{|x_{i_{k}}-x_{0}|}=\frac{\theta}{|\theta|}\in \mbox{\rm Im}\,X_{z}(t_{0},z_{0})$, which leads to a contradiction. So we have $X_{z}(t_{0},z_{0})\nu=0$. This implies $\lim_{k\to\infty}\frac{x_{i_{k}}-x_{0}}{|z_{i_{k}}-z_{0}|}=0$ and $\lim_{k\to\infty}\frac{p_{i_{k}}-p_{0}}{|z_{i_{k}}-z_{0}|}=P_{z}(t_{0},z_{0})\nu\neq0$. Therefore,
\begin{align*}
\lim_{k\to\infty}\frac{|p_{i_{k}}-p_{0}|}{|x_{i_{k}}-x_{0}|}=\lim_{k\to\infty}\frac{|p_{i_{k}}-p_{0}|}{|z_{i_{k}}-z_{0}|}\cdot\frac{|z_{i_{k}}-z_{0}|}{|x_{i_{k}}-x_{0}|}
=|P_{z}(t_{0},z_{0})\nu|\cdot+\infty=+\infty.
\end{align*}
Since $\{z_{i'}\}$ is arbitrary, we obtain that $\lim_{i\to \infty}|\frac{p_{i}-p_{0}}{x_{i}-x_{0}}|=+\infty$.
\end{proof}

\begin{Pro}
$\mu(\Sigma)=\mu(\Gamma)=0$, where $\mu(\cdot)$ denotes the Lebesgue's measure of a subset of $\R^{n+1}$.
\end{Pro}

\begin{proof}
$u\in\mbox{\rm SCL}_{loc}\,((0,+\infty]\times\R^{n})$ implies $\mu(\Sigma)=0$. Fix any $t\in(0,+\infty)$, $X(t,\cdot)$ is a $C^{R}$ map on $\R^{n}$. By Sard's Theorem, we have $\mu(\{X(t,z)|z\in\R^{n},\det{X_{z}(t,z)}=0\})=0$. Using Fubini's Theorem, we know that $\mu(\{X(t,z)|(t,z)\in(0,+\infty)\times\R^{n},\det{X_{z}(t,z)}=0\})=0$. Therefore,  $\mu(\Gamma)=0$.
\end{proof}

\section{Strict singular characteristics from and to non-conjugate points}

In this section, we will analyze the local structure of strict singular characteristics from and to a non-conjugate point. This approach gives us more information on the propagation of singularities especially with geometric intuition.


Recall some basic notions from convex analysis. For any compact convex subset $K$ of $\R^n$ and $\theta\in\R^n\setminus\{0\}$, the \emph{exposed face of $K$ in direction $\theta$} is the set
	\begin{align*}
		\text{epf}\,(K,\theta)=\{x\in K: \langle y-x,\theta\rangle\geqslant0\ \text{for all}\ y\in K\}.
	\end{align*}
A set $\{x_0, ... , x_k\}$ of points of $\R^n$ is said to be \emph{geometrically independent}, or affinely independent, if the equations $\sum_{i=0}^ka_ix_i=0$ and $\sum_{i=0}^ka_i=0$ hold only if each $a_i=0$.

Fix $(t_0,x_0)\in\Sigma\setminus\Gamma$. By Proposition \ref{irregu non-conj pro} and Corollary \ref{cor:conj closed}, there exists a neighborhood $N\subset\Gamma^{c}$ of $(t_0,x_0)$ and $v_i\in C^{R+1}(N,\R)$, $i=1,\cdots,k,\ k\geqslant 2$ such that for all $(t,x)\in N$, 
\begin{equation}\label{eq:non_conj1}
	\begin{split}
		u(t,x)=&\,\min_{1\leqslant i\leqslant k}\{v_{i}(t,x)\},\\
		D^{*}u(t,x)=&\,\{Dv_i(t,x):v_{i}(t,x)=u(t,x)\},
	\end{split}
\end{equation}
and $v_1(t_0,x_0)=\cdots=v_k(t_0,x_0)$. Since we shall handle certain exposed face of $D^+u(t,x)$, in general we define
\begin{equation}\label{eq:D2}
	D^{\#}u(t,x)=\mbox{\rm co}\,\{Dv_i(t,x)\}_{i=1}^{k'},\qquad (t,x)\in N,
\end{equation}
where $\{v_i\}$ is a family of smooth functions in \eqref{eq:non_conj1}. 


\subsection{Minimal energy elements on various exposed faces}

In this section, we shall deal with the exposed face containing the unique minimal energy element $(\bar{q}(t,x),\bar{p}(t,x))$ of the convex function $(q,p)\mapsto q+H(t,x,p,u(t,x))$ on some exposed face of $D^+u(t,x)$ in the form $D^{\#}u(t,x)$. Thus, we recall the following well known result of convex optimization.

\begin{Pro}\label{pro:min_convex}
Let $f$ be a convex function on $\R^n$ and let $C\subset\R^n$ be a convex set and consider the following minimizing problem:
\begin{equation}\label{eq:min_convex}
	\mbox{\rm Min}\,\, f\quad\mbox{\rm subject to}\quad x\in C.
\end{equation}
Then, $x$ be a minimizer of \eqref{eq:min_convex} if and only if\, $0\in\partial f(x)+N_C(x)$, where
$$
N_C(x)=\{\theta\in\R^n: \langle\theta,y-x\rangle\leqslant0, y\in C\},\quad x\in C.
$$
\end{Pro}

As a direct consequence of Proposition \ref{pro:min_convex} together with the structure of $D^{\#}u(t,x)$ and our assumption that $H$ is strictly convex in $p$-variable, we conclude

\begin{Lem}\label{lem:fundamental structure}
We assume $(t_0,x_0)\in\Sigma\setminus\Gamma$ satisfies \eqref{eq:non_conj1} and $D^{\#}u(t,x)$ is defined as in \eqref{eq:D2}. Then, for any $(t,x)\in N$
\begin{enumerate}[\rm (a)]
  \item There exists a unique $(\bar{q}(t,x),\bar{p}(t,x))\in D^{\#}u(t,x)$ such that:
\begin{align*}
\bar{q}(t,x)+H(t,x,\bar{p}(t,x),v_{k'}(t,x))\leqslant q+H(t,x,p,v_{k'}(t,x)),\qquad \forall (q,p)\in D^{\#}u(t,x).
\end{align*}
  \item Set $\bar{v}(t,x)=H_p(t,x,\bar{p}(t,x),v_{k'}(t,x))$. Then we have
\begin{align*}
(\bar{q}(t,x),\bar{p}(t,x))\in\mbox{\rm epf}\,(D^{\#}u(t,x),(1,\bar{v}(t,x))).
\end{align*}
\item If $(q_1,p_1),(q_2,p_2)\in \mbox{\rm epf}\,(D^{\#}u(t,x),(1,\bar{v}(t,x)))$ and $p_1=p_2$, then $q_1=q_2$.
\end{enumerate}
\end{Lem} 

Now, we formulate the main result in this section. 

\begin{The}\label{pro:fundamental structure}
Under the assumption of Lemma \ref{lem:fundamental structure}, if $(t_0,x_0)\in N$ satisfies:
\begin{enumerate}[\rm (i)]
  \item $v_1(t_0,x_0)=\cdots=v_{k'}(t_0,x_0)$,
  \item The vectors in $\{Dv_i(t_0,x_0)\}_{i=1}^{k'}$ are geometrically independent,
  \item $(\bar{q}(t_0,x_0),\bar{p}(t_0,x_0))\in\mbox{\rm ri}\,(D^{\#}u(t_0,x_0))$,
\end{enumerate}
then there exists $\delta>0$ such that
\begin{equation}\label{eq:strict chara 0}
\begin{cases}
&\dot{x}(t)=\bar{v}(t,x(t)),\qquad t\in[t_0-\delta,t_0+\delta],\\
&x(t_0)=x_0,
\end{cases}
\end{equation}
has a unique $C^{R+1}$ solution $x:[t_0-\delta,t_0+\delta]\to\R^n$ and for all $t\in[t_0-\delta,t_0+\delta]$
\begin{enumerate}[\rm (1)]
	\item $v_1(t,x(t))=\cdots=v_{k'}(t,x(t))$,
	\item The vectors in $\{Dv_i(t,x(t))\}_{i=1}^{k'}$ are geometrically independent,
	\item $(\bar{q}(t,x(t)),\bar{p}(t,x(t)))\in\mbox{\rm ri}\,(D^{\#}u(t,x(t)))$,
	\item $\mbox{\rm epf}\,(D^{\#}u(t,x(t)),(1,\bar{v}(t,x(t))))=D^{\#}u(t,x(t))$.
\end{enumerate}
\end{The}

The proof of Theorem \ref{pro:fundamental structure} is based on the following Lemma \ref{lem:linear idp}, Lemma \ref{lem:minimal H 2} and Lemma \ref{lem:vector C R}.

\begin{Lem}\label{lem:linear idp}
Under the assumptions of Theorem \ref{pro:fundamental structure}, there exists a neighborhood $N_1\subset N$ of $(t_0,x_0)$ such that for all $(t,x)\in N_1$,
\begin{enumerate}[\rm (1)]
  \item $\{Dv_i(t,x)\}_{i=1}^{k'}$ are geometrically independent.
  \item $\{\nabla v_i(t,x)-\nabla v_{k'}(t,x):1\leqslant i\leqslant k'-1\}$ are linearly independent.
\end{enumerate}
\end{Lem}

\begin{proof}
Since $\{Dv_i(t_0,x_0)\}_{i=1}^{k'}$ are geometrically independent, we have
\begin{equation}\label{pfeq:linear idp 1}
\mbox{\rm rank}\,(Dv_1(t_0,x_0)-Dv_{k'}(t_0,x_0),\cdots,Dv_{k'-1}(t_0,x_0)-Dv_{k'}(t_0,x_0))=k'-1.
\end{equation}
If $\mbox{\rm rank}\,(\nabla v_1(t_0,x_0)-\nabla v_{k'}(t_0,x_0),\cdots,\nabla v_{k'-1}(t_0,x_0)-\nabla v_{k'}(t_0,x_0))<k'-1$, then there exists $\alpha_i\in\R,\ 1\leqslant i\leqslant k'-1$ such that the $\alpha_i$'s are not all $0$ and
\begin{equation}\label{pfeq:linear idp 2}
	\sum_{i=1}^{k'-1}\alpha_i(\nabla v_i(t_0,x_0)-\nabla v_{k'}(t_0,x_0))=0.
\end{equation}
By \eqref{pfeq:linear idp 1} and \eqref{pfeq:linear idp 2}, we obtain
\begin{equation}\label{pfeq:linear idp 3}
	\sum_{i=1}^{k'-1}\alpha_i(\frac{\partial}{\partial t}v_i(t_0,x_0)-\frac{\partial}{\partial t}v_{k'}(t_0,x_0))\neq 0.
\end{equation}
On the other hand, due to Lemma \ref{lem:fundamental structure} we obtain
\begin{align*}
	\mbox{\rm epf}\,(D^{\#}u(t_0,x_0),(1,\bar{v}(t_0,x_0)))=D^{\#}u(t_0,x_0).
\end{align*}
It follows that for all $1\leqslant i,j\leqslant k'$,
\begin{equation}\label{pfeq:linear idp 4}
	\frac{\partial}{\partial t}v_i(t_0,x_0)+\bar{v}(t_0,x_0)\cdot\nabla v_i(t_0,x_0)=\frac{\partial}{\partial t}v_{j}(t_0,x_0)+\bar{v}(t_0,x_0)\cdot\nabla v_{j}(t_0,x_0).
\end{equation}
Combing \eqref{pfeq:linear idp 2} with \eqref{pfeq:linear idp 4}, we have
\begin{align*}
	&\,\sum_{i=1}^{k'-1}\alpha_i(\frac{\partial}{\partial t}v_i(t_0,x_0)-\frac{\partial}{\partial t}v_{k'}(t_0,x_0))\\
   =&\,\sum_{i=1}^{k'-1}\alpha_i(\frac{\partial}{\partial t}v_i(t_0,x_0)-\frac{\partial}{\partial t}v_{k'}(t_0,x_0))+\bar{v}(t_0,x_0)\sum_{i=1}^{k'-1}\alpha_i(\nabla v_i(t_0,x_0)-\nabla v_{k'}(t_0,x_0))\\
   =&\,\sum_{i=1}^{k'-1}\alpha_i(\frac{\partial}{\partial t}v_i(t_0,x_0)+\bar{v}(t_0,x_0)\cdot\nabla v_i(t_0,x_0)-\frac{\partial}{\partial t}v_{k'}(t_0,x_0)-\bar{v}(t_0,x_0)\cdot\nabla v_{k'}(t_0,x_0))\\
   =&\,0,
\end{align*}
which leads to a contradiction to \eqref{pfeq:linear idp 3}. Therefore,
\begin{align*}
	\mbox{\rm rank}\,(\nabla v_1(t_0,x_0)-\nabla v_{k'}(t_0,x_0),\cdots,\nabla v_{k'-1}(t_0,x_0)-\nabla v_{k'}(t_0,x_0))=k'-1.
\end{align*}
Noticing that the $v_i$'s are of class $C^{R+1}$ in $N$, there exists a neighborhood $N_1\subset N$ of $(t_0,x_0)$ such that (1) and (2) hold, by implicit function theorem.
\end{proof}

Now we define a convex set $\Delta\subset\R^{k'-1}$ in the form
\begin{align*}
\Delta=\{\lambda=(\lambda_1,\cdots,\lambda_{k'-1})\in \R^{k'-1}:0\leqslant\lambda_i\leqslant 1,\ i=1,\cdots,k'-1,\ 0\leqslant 1-\sum_{i=1}^{k'-1}\lambda_i\leqslant 1\}.
\end{align*}
Then for any $(t,x)\in N_1$, by Lemma \ref{lem:linear idp}, the map
\begin{align*}
\Phi_{t,x}:\Delta & \to D^{\#}u(t,x)\\
\lambda=(\lambda_1,\cdots,\lambda_{k'-1}) & \mapsto \Phi_{t,x}(\lambda)=(\bar{q}(t,x,\lambda),\bar{p}(t,x,\lambda)).
\end{align*}
is a linear isomorphism, where
\begin{align*}
\bar{q}(t,x,\lambda)=\sum_{i=1}^{k'-1}\lambda_{i}(\frac{\partial}{\partial t}v_i(t,x)-\frac{\partial}{\partial t}v_{k'}(t,x))+\frac{\partial}{\partial t}v_{k'}(t,x),\\
\bar{p}(t,x,\lambda)=\sum_{i=1}^{k'-1}\lambda_{i}(\nabla v_i(t,x)-\nabla v_{k'}(t,x))+\nabla v_{k'}(t,x).
\end{align*}
We define
\begin{align*}
	E(t,x,\lambda)=\bar{q}(t,x,\lambda)+H(t,x,\bar{p}(t,x,\lambda),v_{k'}(t,x)),\qquad (t,x,\lambda)\in N_1\times\Delta.
\end{align*}
Obviously, $\bar{q},\ \bar{p}$ and $E$ are of class $C^R$ in $N_1\times\Delta$ and $E$ strictly convex with variable $\lambda$.

\begin{Lem}\label{lem:minimal H 2}
Suppose $(t,x)\in N_1$, then
\begin{enumerate}[\rm (a)]
  \item There exists a unique $\bar{\lambda}(t,x)\in \Delta$ such that
  \begin{align*}
E(t,x,\bar{\lambda}(t,x))\leqslant E(t,x,\lambda),\qquad \forall \lambda\in\Delta.
  \end{align*}
and we have $\Phi_{t,x}(\bar{\lambda}(t,x))=(\bar{q}(t,x),\bar{p}(t,x))$.
  \item $\bar{\lambda}(t,x)\in\mbox{\rm int}\,(\Delta)$\footnote{We denote by $\mbox{\rm int}\,(C)$ the interior of a set $C$.} if and only if $(\bar{q}(t,x),\bar{p}(t,x))\in\mbox{\rm ri}\,(D^{\#}u(t,x))$.
  \item If $\lambda\in\Delta$ and $D_{\lambda}E(t,x,\lambda)=0$, then we have $\lambda=\bar{\lambda}(t,x)$. If $\bar{\lambda}(t,x)\in\mbox{\rm int}\,(\Delta)$, then we have $D_{\lambda}E(t,x,\lambda(t,x))=0$.
\end{enumerate}
\end{Lem}

\begin{proof}
By Lemma \ref{lem:fundamental structure} (a) and the linear isomorphism $\Phi_{t,x}$, we obtain (a) and (b) directly. The consequence (c) is trivial since $E$ is of class $C^R$ and convex with respect to $\lambda$.
\end{proof}

\begin{Lem}\label{lem:vector C R}
There exists a neighborhood $N_2\subset N_1$ of $(t_0,x_0)$ such that:
\begin{enumerate}[\rm (1)]
  \item The function $(t,x)\mapsto\bar{\lambda}(t,x)$ is of class $C^{R}$ in $N_2$, and $\bar{\lambda}(t,x)\in\mbox{\rm int}\,(\Delta)$ for all $(t,x)\in N_2$.
  \item $\bar{q}(t,x),\ \bar{p}(t,x)$ and $\bar{v}(t,x)$ are of class $C^{R}$ in $N_2$, and $(\bar{q}(t,x),\bar{p}(t,x))\in\mbox{\rm ri}\,(D^{\#}u(t,x))$ for all $(t,x)\in N_2$.
  \item For all $(t,x)\in N_2$, $\mbox{\rm epf}\,(D^{\#}u(t,x),(1,\bar{v}(t,x)))=D^{\#}u(t,x)$.
\end{enumerate}
\end{Lem}

\begin{proof}
Since $(\bar{q}(t_0,x_0),\bar{p}(t_0,x_0))\in\mbox{\rm ri}\,(D^{\#}u(t_0,x_0))$, we have $\bar{\lambda}(t_0,x_0)\in\mbox{\rm int}\,(\Delta)$ by Lemma \ref{lem:minimal H 2} (b). According to Lemma \ref{lem:minimal H 2} (c), there holds
\begin{align*}
D_{\lambda}E(t_0,x_0,\bar{\lambda}(t_0,x_0))=0.
\end{align*}
Differentiating $E(t,x,\lambda)$ by $\lambda$, we obtain
\begin{equation*}
	\begin{split}
		D_{\lambda}E(t,x,\lambda)&=V_t(t,x)+H_p(t,x,\bar{p}(t,x,\lambda),v_{k'}(t,x))\nabla V(t,x),\\
		D^{2}_{\lambda\lambda}E(t,x,\lambda)&=(\nabla V(t,x))^{T}H_{pp}(t,x,\bar{p}(t,x,\lambda),v_{k'}(t,x))\nabla V(t,x),
	\end{split}
	\qquad (t,x,\lambda)\in N_1 \times \Delta.
\end{equation*}
where
\begin{equation*}
	\begin{split}
		V_t(t,x)&=(\frac{\partial}{\partial t}v_1(t,x)-\frac{\partial}{\partial t}v_{k'}(t,x),\cdots,\frac{\partial}{\partial t}v_{k'-1}(t,x)-\frac{\partial}{\partial t}v_{k'}(t,x))\\
		\nabla V(t,x)&=(\nabla v_1(t,x)-\nabla v_{k'}(t,x),\cdots,\nabla v_{k'-1}(t,x)-\nabla v_{k'}(t,x)),
	\end{split}
	\qquad (t,x,\lambda)\in N_1 \times \Delta.
\end{equation*}
Notice that $H_{pp}$ is positive definite and Lemma \ref{lem:linear idp} (2) implies $\mbox{\rm rank}\,(\nabla V(t,x))=k'-1$. Therefore, $D^{2}_{\lambda\lambda}E(t,x,\lambda)$ is non-degenerate for all $(t,x,\lambda)\in N_1 \times \Delta$. Now, by the implicit function theorem, there exists a neighborhood $N_2\subset N_1$ of $(t_0,x_0)$ and a $C^R$ function $\tilde{\lambda}:N_2\to \mbox{\rm int}\,(\Delta)$ such that $\tilde{\lambda}(t_0,x_0)=\bar{\lambda}(t_0,x_0)$ and
\begin{align*}
D_{\lambda}E(t,x,\tilde{\lambda}(t,x))=0,\qquad \forall (t,x)\in N_2.
\end{align*}
Combing this with Lemma \ref{lem:minimal H 2} (b), we have $\tilde{\lambda}(t,x)=\bar{\lambda}(t,x)$ for all $(t,x)\in N_2$, which leads to our conclusion (1). Conclusion (2) follows from Lemma \ref{lem:minimal H 2} (a) (b) and Lemma \ref{lem:vector C R} (1), and conclusion (3) Proposition \ref{lem:fundamental structure} (b) and Lemma \ref{lem:vector C R} (2).
\end{proof}

\begin{proof}[Proof of Theorem \ref{pro:fundamental structure}]
Now consider \eqref{eq:strict chara 0}. By Lemma \ref{lem:vector C R} (2) and the Cauchy-Lipschitz theorem, there exists $\delta>0$ such that \eqref{eq:strict chara 0} has a unique $C^{R+1}$ solution $x:[t_0-\delta,t_0+\delta]\to\R^n$ and $(t,x(t))\in N_2$ for all $t\in[t_0-\delta,t_0+\delta]$. For $1\leqslant i\leqslant k'$, Lemma \ref{lem:vector C R} (3) implies
\begin{align*}
\langle Dv_i(t,x(t)),(1,\bar{v}(t,x(t))) \rangle=\langle Dv_{k'}(t,x(t)),(1,\bar{v}(t,x(t))) \rangle,\qquad \forall t\in[t_0-\delta,t_0+\delta].
\end{align*}
Then we have that for any $t\in[t_0,t_0+\delta]$,
\begin{align*}
v_i(t,x(t))&=v_i(t_0,x_0)+\int_{t_0}^{t} \langle Dv_i(s,x(s)),(1,\dot{x}(s)) \rangle ds\\
           &=v_i(t_0,x_0)+\int_{t_0}^{t} \langle Dv_i(s,x(s)),(1,\bar{v}(s,x(s))) \rangle ds\\
           &=v_{k'}(t_0,x_0)+\int_{t_0}^{t} \langle Dv_1(s,x(s)),(1,\bar{v}(s,x(s))) \rangle ds\\
           &=v_{k'}(t,x(t)).
\end{align*}
Similarly, there holds
\begin{align*}
v_i(t,x(t))=v_{k'}(t,x(t)),\qquad \forall t\in[t_0-\delta,t_0].
\end{align*}
This completes the proof of (1) in Theorem \ref{pro:fundamental structure}. Conclusion (2), (3) and (4) are direct consequences of Lemma \ref{lem:linear idp} (1), Lemma \ref{lem:vector C R} (2) and Lemma \ref{lem:vector C R} (3).
\end{proof}

\subsection{Smooth strict singular characteristics from non-conjugate singular points}

Fix $(t_0,x_0)\in\Sigma\setminus\Gamma$ such that \eqref{eq:non_conj1} holds.  Without loss of generality, let
\begin{align*}
D^{\#}u(t_0,x_0)=\mbox{\rm co}\,\{Dv_i(t_0,x_0)\}_{i=1}^{k'}=\mbox{\rm epf}\,(D^{+}u(t_0,x_0),(1,v(t_0,x_0))),
\end{align*}
where $v(t_0,x_0)=\bar{v}(t_0,x_0)$ is defined in Lemma \ref{lem:fundamental structure} (b) in the case that $D^{\#}u(t_0,x_0)=D^{+}u(t_0,x_0)$. That is, $v(t_0,x_0)=H_p(t_0,x_0,p(t_0,x_0),u(t_0,x_0))$ where $(q(t_0,x_0),p(t_0,x_0))\in D^{+}u(t_0,x_0)$ is the unique minimal energy element of the function $(q,p)\mapsto q+H(t_0,x_0,p,u(t_0,x_0))$, and $D^{\#}u(t_0,x_0)$ is the exposed face of $D^{+}u(t_0,x_0)$ containing such a minimal energy element. 

\begin{Rem}
When considering propagation of singularities for stationary equations such as
\begin{equation}\label{eq:HJs}\tag{HJ$_s$}
H(x,Du(x),u(x))=0,\ x\in \R^{n},
\end{equation}
we have to exclude the case $0\in \mbox{\rm co}\, H_{p}(x,D^{+}u(x),u(x))$. But for evolutionary equations \eqref{eq:HJe} here, it will not happen because the first variable of $(1,v(t,x))$ is always not $0$.
\end{Rem}

\begin{defn}
We call $(t_0,x_0)\in\Sigma\setminus\Gamma$ is \emph{non-degenerate} if
\begin{enumerate}[(i)]
	\item The vectors in $\{Dv_i(t_0,x_0)\}_{i=1}^{k'}$ are geometrically independent.
	\item $(q(t_0,x_0),p(t_0,x_0))\in\mbox{\rm ri}\,(D^{\#}(t_0,x_0))$.
\end{enumerate}
\end{defn}

\begin{The}\label{thm:local propagation positive}
Suppose $(t_0,x_0)\in\Sigma\setminus\Gamma$ is non-degenerate. Then there exists $\delta_0>0$ such that
\begin{equation}\label{eq:strict chara positive}
\begin{cases}
&\dot{x}(t)=v(t,x(t)),\qquad t\in[t_0,t_0+\delta_0],\\
&x(t_0)=x_0,
\end{cases}
\end{equation}
has a unique $C^{R+1}$ solution $x:[t_0,t_0+\delta_0]\to \R^n$ and
\begin{align*}
(t,x(t))\in\Sigma\setminus\Gamma,\qquad \forall t\in[t_0,t_0+\delta_0].
\end{align*}
\end{The}

\begin{proof}
Lemma \ref{lem:fundamental structure} (b) implies $(q(t_0,x_0),p(t_0,x_0))=(\bar{q}(t_0,x_0),\bar{p}(t_0,x_0))$ and $v(t_0,x_0)=\bar{v}(t_0,x_0)$. For non-degenerate $(t_0,x_0)\in\Sigma\setminus\Gamma$, by Theorem \ref{pro:fundamental structure}, there exists $\delta>0$ such that \eqref{eq:strict chara 0} has a unique $C^{R+1}$ solution $x:[t_0-\delta,t_0+\delta]\to\R^n$ and
\begin{align*}
v_1(t,x(t))=\cdots=v_{k'}(t,x(t)),\qquad \forall t\in[t_0-\delta,t_0+\delta].
\end{align*}
Since $D^{\#}u(t_0,x_0)=\mbox{\rm epf}\,(D^{+}u(t_0,x_0),(1,v(t_0,x_0)))$, there exists $\sigma>0$ such that
\begin{align*}
\langle Dv_i(t_0,x_0),(1,\bar{v}(t_0,x_0))\rangle\geqslant\langle Dv_{k'}(t_0,x_0),(1,\bar{v}(t_0,x_0))\rangle+\sigma,\qquad \forall k'+1\leqslant i\leqslant k.
\end{align*}
Obviously, $x$ and the $v_i$'s are of class $C^{R+1}$ and $\bar{v}(\cdot,x(\cdot))=\dot{x}(\cdot)$ is of class $C^R$. Thus, there exists $0<\delta_0\leqslant\delta$ such that
\begin{equation}\label{pfeq:thm from 1}
\langle Dv_i(t,x(t)),(1,\bar{v}(t,x(t)))\rangle\geqslant\langle Dv_{k'}(t,x(t)),(1,\bar{v}(t,x(t)))\rangle+\frac{1}{2}\sigma
\end{equation}
for all $k'+1\leqslant i\leqslant k$ and $t\in[t_0,t_0+\delta_0]$. It follows that
\begin{align*}
v_i(t,x(t))&=v_i(t_0,x_0)+\int_{t_0}^{t} \langle Dv_i(s,x(s)),(1,\dot{x}(s)) \rangle ds\\
           &=v_i(t_0,x_0)+\int_{t_0}^{t} \langle Dv_i(s,x(s)),(1,\bar{v}(s,x(s))) \rangle ds\\
           &\geqslant v_{k'}(t_0,x_0)+\int_{t_0}^{t} \big\{\langle Dv_{k'}(s,x(s)),(1,\bar{v}(s,x(s))) \rangle+\frac{1}{2}\sigma\}ds\\
           &=v_{k'}(t,x(t))+\frac{1}{2}\sigma(t-t_0)\\
           &>v_{k'}(t,x(t))
\end{align*}
for all $k'+1\leqslant i\leqslant k$ and $t\in(t_0,t_0+\delta_0]$. We conclude that for $t\in(t_0,t_0+\delta_0]$,
\begin{equation}\label{pfeq:thm from 2}
u(t,x(t))=\min_{1\leqslant i\leqslant k}v_i(t,x(t))=v_{j_1}(t,x(t))<v_{j_2}(t,x(t)),\ \forall 1\leqslant j_1\leqslant k',k'+1\leqslant j_2\leqslant k
\end{equation}
and $D^{+}u(t,x(t))=D^{\#}u(t,x(t))$. Therefore, when $t\in(t_0,t_0+\delta_0]$, there holds
\begin{equation}\label{pfeq:thm from 3}
	\begin{split}
		(\bar{q}(t,x(t)),\bar{p}(t,x(t)))&=\argmin_{(q,p)\in D^{\#}u(t,x(t))}\{q+H(t,x(t),p,v_{k'}((t,x(t))))\}\\
		&=\argmin_{(q,p)\in D^{+}u(t,x(t))}\{q+H(t,x(t),p,u((t,x(t))))\}\\
		&=(q(t,x(t)),p(t,x(t))).
	\end{split}
\end{equation}
and
\begin{equation}\label{pfeq:thm from 4}
\bar{v}(t,x(t))=H_p(t,x(t),\bar{p}(t,x(t)),v_{k'}(t,x(t)))=H_p(t,x(t),p(t,x(t)),u(t,x(t)))=v(t,x(t)).
\end{equation}
This implies $x:[t_0,t_0+\delta_0]\to \R^n$ is a $C^{R+1}$ solution of \eqref{eq:strict chara positive} and
\begin{align*}
(t,x(t))\in\Sigma\setminus\Gamma,\qquad \forall t\in[t_0,t_0+\delta_0].
\end{align*}

Next, we turn to the uniqueness. Let $\gamma:[t_0,t_0+\delta_0]\to \R^n$ be a Lipschitz solution of \eqref{eq:strict chara positive} such that $\dot{\gamma}^{+}(t)=v(t,\gamma(t))$ and $\dot{\gamma}^{+}(\cdot)$ is right-continuous on $[t_0,t_0+\delta_0]$. Set
\begin{align*}
	\tau=\sup\{t:\gamma(s)=x(s),\ \forall s\in[t_0,t]\}.
\end{align*}
Obviously, we have $t_0\leqslant\tau\leqslant t_0+\delta_0$, $\gamma(\tau)=x(\tau)$ and \eqref{pfeq:thm from 1}, \eqref{pfeq:thm from 2}, \eqref{pfeq:thm from 3}, \eqref{pfeq:thm from 4} implies
\begin{equation}\label{pfeq:thm from 5}
(q(\tau,\gamma(\tau)),p(\tau,\gamma(\tau)))=(\bar{q}(\tau,\gamma(\tau)),\bar{p}(\tau,\gamma(\tau))),\ v(\tau,\gamma(\tau))=\bar{v}(\tau,\gamma(\tau)),
\end{equation}
\begin{equation}\label{pfeq:thm from 6}
v_i(\tau,\gamma(\tau))\geqslant v_1(\tau,\gamma(\tau)),\qquad \forall k'+1\leqslant i\leqslant k,
\end{equation}
\begin{equation}\label{pfeq:thm from 7}
\langle Dv_i(\tau,\gamma(\tau)),(1,\bar{v}(\tau,\gamma(\tau)))\rangle\geqslant\langle Dv_1(\tau,\gamma(\tau)),(1,\bar{v}(\tau,\gamma(\tau)))\rangle+\frac{1}{2}\sigma,\qquad \forall k'+1\leqslant i\leqslant k.
\end{equation}
Now, we claim $\tau=t_0+\delta_0$. Otherwise, we assume $\tau<t_0+\delta_0$. Since $\gamma$ is Lipschitz, there exists $\tau<t'\leqslant t_0+\delta_0$ such that $\gamma(t)\in N$ for all $t\in[\tau,t']$. By \eqref{pfeq:thm from 6}, \eqref{pfeq:thm from 7} and a similar argument as above, there exists $\tau<t''\leqslant t'$ such that
\begin{align*}
v_i(t,\gamma(t))>v_{k'}(t,\gamma(t)),\qquad \forall k'+1\leqslant i\leqslant k,\ t\in(\tau,t''],
\end{align*}
which implies
\begin{equation}\label{pfeq:thm from 8}
(q(t,\gamma(t)),p(t,\gamma(t)))\in D^{+}u(t,\gamma(t))\subset D^{\#}u(t,\gamma(t)),\qquad \forall t\in(\tau,t''].
\end{equation}
It follows from Theorem \ref{pro:fundamental structure} (2) (3) (4) and \eqref{pfeq:thm from 5} that the vectors in $\{Dv_i(\tau,\gamma(\tau))\}_{i=1}^{k'}$ are geometrically independent and
\begin{align*}
&(q(\tau,\gamma(\tau)),p(\tau,\gamma(\tau)))\in\mbox{\rm ri}\,(D^{\#}u(\tau,\gamma(\tau))),\\
&\mbox{\rm epf}\,(D^{\#}u(\tau,\gamma(\tau)),(1,\bar{v}(\tau,\gamma(\tau))))=D^{\#}u(\tau,\gamma(\tau)).
\end{align*}
By Lemma \ref{lem:fundamental structure} (c), there exists $\sigma'>0$ such that
\begin{align*}
|p(\tau,\gamma(\tau))-p|\geqslant\sigma',\qquad \forall (q,p)\in D^{\#}u(\tau,\gamma(\tau))\setminus \mbox{\rm ri}\,(D^{\#}u(\tau,\gamma(\tau))).
\end{align*}
It is easy to see that $p(t,\gamma(t))=L_v(t,\gamma(t),v(t,\gamma(t)),u(t,\gamma(t)))$ and the $Dv_i(t,\gamma(t))$'s are right-continuous. Thus, there exists $t_0<t'''\leqslant t''$ such that for all $t\in[\tau,t''']$, the vectors in $\{Dv_i(t,\gamma(t))\}_{i=1}^{k'}$ are geometrically independent and
\begin{equation}\label{pfeq:thm from 9}
|p(t,\gamma(t))-p|\geqslant\frac{1}{2}\sigma',\qquad \forall (q,p)\in D^{\#}u(t,\gamma(t))\setminus \mbox{\rm ri}\,(D^{\#}u(t,\gamma(t))),
\end{equation}
Combing \eqref{pfeq:thm from 8} with \eqref{pfeq:thm from 9}, it follows that
\begin{equation}\label{pfeq:thm from 10}
(q(t,\gamma(t)),p(t,\gamma(t)))\in\mbox{\rm ri}\,(D^{\#}u(t,\gamma(t))),\qquad \forall t\in[\tau,t'''].
\end{equation}
Invoking \eqref{pfeq:thm from 8}, \eqref{pfeq:thm from 10} and the geometrical independence of $\{Dv_i(t,\gamma(t))\}_{i=1}^{k'}$, we obtain that
\begin{align*}
D^{+}u(t,\gamma(t))=D^{\#}u(t,\gamma(t)),\qquad \forall t\in[\tau,t'''],
\end{align*}
which implies
\begin{align*}
v_1(t,\gamma(t))=\cdots=v_{k'}(t,\gamma(t))=u(t,\gamma(t)),\qquad \forall t\in[\tau,t'''].
\end{align*}
Therefore, we have that for all $t\in(\tau,t''']$,
\begin{align*}
(q(t,\gamma(t)),p(t,\gamma(t)))&=\argmin_{(q,p)\in D^{+}u(t,\gamma(t))}q+H(t,\gamma(t),p,u(t,\gamma(t)))\\
                               &=\argmin_{(q,p)\in D^{\#}u(t,\gamma(t))}q+H(t,\gamma(t),p,v_{k'}(t,x))\\
                               &=(\bar{q}(t,\gamma(t)),\bar{p}(t,\gamma(t))),
\end{align*}
and
\begin{align*}
v(t,\gamma(t))&=H_p(t,\gamma(t),p(t,\gamma(t)),u(t,\gamma(t)))\\
              &=H_p(t,\gamma(t),\bar{p}(t,\gamma(t)),v_{k'}(t,\gamma(t)))\\
              &=\bar{v}(t,\gamma(t)).
\end{align*}
Now, we conclude that $\gamma:[t_0,t''']\to \R^n$ is in fact a Lipschitz solution of \eqref{eq:strict chara 0}. By the Cauchy-Lipschitz theorem, there holds $\gamma(t)=x(t)$ for all $t\in[t_0,t''']$. This leads to a contradiction with $\tau<t'''$. Finally, we have $\tau=t_0+\delta_0$ and $\gamma(t)=x(t)$ for all $t\in[t_0,t_0+\delta_0]$. This completes the proof uniqueness.
\end{proof}

\subsection{Smooth strict singular characteristics to non-conjugate singular points}

In this section, we will deal with the possible exposed face $D^{\#}u(t,x)$ of $D^+u(t,x)$ where the a minimax energy element is located. 

\begin{defn}
Suppose $(t_0,x_0)\in\Sigma\setminus\Gamma$, we call $(q_0,p_0)\in D^{+}u(t_0,x_0)$ is \emph{minimax} at $(t_0,x_0)$ if
	\begin{align*}
		(q_0,p_0)\in \mbox{\rm epf}\,(D^{+}u(t_0,x_0),-(1,H_p(t_0,x_0,p_0,u(t_0,x_0))))\setminus D^{*}u(t_0,x_0).
	\end{align*}
\end{defn}

Without loss of generality, set
\begin{align*}
    D^{\#}u(t_0,x_0)=\mbox{\rm co}\,\{Dv_i(t_0,x_0)\}_{i=1}^{k'}=\mbox{\rm epf}\,(D^{+}u(t_0,x_0),-(1,H_p(t_0,x_0,p_0,u(t_0,x_0)))),
\end{align*}
where $k'\geqslant 2$. Now we have
\begin{Lem}\label{lem:mini-max}
	\begin{align*}
	(q_0,p_0)=(\bar{q}(t_0,x_0),\bar{p}(t_0,x_0)),\ H_p(t_0,x_0,p_0,u(t_0,x_0))=\bar{v}(t_0,x_0)
	\end{align*}
where $(\bar{q},\bar{p})$ and $\bar{v}$ are given in Lemma \ref{lem:fundamental structure}.
\end{Lem}

\begin{Rem}
This lemma ensure that, if $(t_0,x_0)\in\Sigma\setminus\Gamma$ and $(q_0,p_0)\in D^{+}u(t_0,x_0)$ is minimax, then $(q_0,p_0)$ is the unique minimal energy element restricted on the exposed face $D^{\#}u(t_0,x_0)$.
\end{Rem}

\begin{proof}
Consider the function
\begin{align*}
f:\R^{n+1}\to \R,\qquad (q,p)\mapsto q+H(t_0,x_0,p,u(t_0,x_0)).
\end{align*}
Obviously, $f$ is convex and of class $C^{R+1}$, and
\begin{align*}
Df(q,p)=(1,H_p(t_0,x_0,p,u(t_0,x_0))),\qquad \forall (q,p)\in\R^{n+1}.
\end{align*}
Since $D^{\#}u(t_0,x_0)=\mbox{\rm epf}\,(D^{+}u(t_0,x_0),-Df(q_0,p_0))$, the projection of $Df(q_0,p_0)$ to the subspace generated by $D^{\#}u(t_0,x_0)$ is $0$. By convexity of $f$ and $D^{\#}u(t_0,x_0)$, we have
\begin{align*}
f(q_0,p_0)\leqslant f(q,p),\qquad \forall (q,p)\in D^{\#}u(t_0,x_0),
\end{align*}
which implies $(q_0,p_0)=(\bar{q}(t_0,x_0),\bar{p}(t_0,x_0))$ and $H_p(t_0,x_0,p_0,u(t_0,x_0))=\bar{v}(t_0,x_0)$.
\end{proof}

\begin{The}\label{thm:local propagation negative}
Suppose $(t_0,x_0)\in\Sigma\setminus\Gamma$, $(q_0,p_0)$ is non-degenerate and minimax at $(t_0,x_0)$. Then, there exists $\delta_0>0$ such that
\begin{equation}\label{eq:strict chara negative}
\begin{cases}
&\dot{x}(t)=v(t,x(t)),\qquad \forall t\in[t_0-\delta_0,t_0),\\
&\dot{x}^{-}(t_0)=H_p(t_0,x_0,p_0,u(t_0,x_0)),\\
&x(t_0)=x_0.
\end{cases}
\end{equation}
has a unique $C^{R+1}$ solution $x:[t_0-\delta_0,t_0]\to \R^n$ and
\begin{align*}
(t,x(t))\in\Sigma\setminus\Gamma,\qquad \forall t\in[t_0-\delta_0,t_0].
\end{align*}
\end{The}

\begin{proof}
In view of Lemma \ref{lem:mini-max} we have $(\bar{q}(t_0,x_0),\bar{p}(t_0,x_0))\in \mbox{\rm ri}\,(D^{\#}u(t_0,x_0))$. Thus, by Theorem \ref{pro:fundamental structure}, there exists $\delta>0$ such that \eqref{eq:strict chara 0} has a unique $C^{R+1}$ solution $x:[t_0-\delta,t_0+\delta]\to\R^n$. In fact, there exists $0<\delta_0\leqslant\delta$ such that $x:[t_0-\delta_0,t_0]\to \R^n$ is the unique $C^{R+1}$ solution of \eqref{eq:strict chara negative}. The rest of the proof is almost the same as that of Theorem \ref{thm:local propagation positive}.
\end{proof}

Now we consider the case $(t_0,x_0)\in\Sigma\setminus\Gamma$ and $k=2$. By Proposition \ref{irregu non-conj pro} and Corollary \ref{cor:conj closed}, there exists a neighborhood $N\subset\Gamma^{c}$ of $(t_0,x_0)$ and $v_1,v_2\in C^{R+1}(N,\R)$ such that
\begin{align*}
	u(t,x)=\min\{v_1(t,x),v_2(t,x)\},\qquad (t,x)\in N,
\end{align*}
and we have $v_1(t_0,x_0)=v_2(t_0,x_0)$.
Since $Dv_1(t_0,x_0)\neq Dv_2(t_0,x_0)$, by the implicit function theorem, there exists a neighborhood $N_0\subset N$ of $(t_0,x_0)$ such that $v_1(t,x)=v_2(t,x)$ determines a $C^{R+1}$ hyper-surface $K^{*}=N_0\cap\Sigma$.

\begin{Cor}\label{cor:local propagation k=2}
Suppose $(t_0,x_0)\in\Sigma\setminus\Gamma$, and $k=2$. Then there exists $\delta_0>0$ such that the equation
\begin{equation}\label{eq:strict chara 3}
\begin{cases}
&\dot{x}(t)=v(t,x(t)),\qquad \forall t\in[t_0-\delta_0,t_0+\delta_0],\\
&x(t_0)=x_0.
\end{cases}
\end{equation}
has a unique $C^{R+1}$ solution $x:[t_0-\delta_0,t_0+\delta_0]\to \R^n$ and
\begin{align*}
(t,x(t))\in\Sigma\setminus\Gamma,\qquad \forall t\in[t_0-\delta_0,t_0+\delta_0].
\end{align*}
\end{Cor}

\begin{proof}
Since $Dv_1(t_0,x_0)\neq Dv_2(t_0,x_0)$ and $(q(t_0,x_0),p(t_0,x_0))\in(Dv_1(t_0,x_0),Dv_2(t_0,x_0))$, the point $(t_0,x_0)\in\Sigma\setminus\Gamma$ is spontaneously non-degenerate. On the other hand, we have
\begin{align*}
[Dv_1(t_0,x_0),Dv_2(t_0,x_0)]=\mbox{\rm epf}\,(D^{+}u(t_0,x_0),-(1,v(t_0,x_0))),
\end{align*}
which implies $(q(t_0,x_0),p(t_0,x_0))$ is a non-degenerate mini-max of $(t_0,x_0)$. Therefore, our conclusion follows directly from Theorem \ref{thm:local propagation positive} and \ref{thm:local propagation negative}.
\end{proof}

\begin{Rem}
The conditions in the main results (Theorem \ref{thm:local propagation positive} and Theorem \ref{thm:local propagation negative}) are satisfied when $n=1$, because any exposed face of $D^+u(t_0,x_0)$ is a segment for $(t_0,x_0)\in\Sigma\setminus\Gamma$.
\end{Rem}

\section{Global propagation of $C^1$ singular support}

For any viscosity solution $u$ of \eqref{eq:HJe}, the \emph{$C^1$ support} (resp. \emph{$C^{1,1}$ support}) of $u$, denoted by $\supp_{C^1}u$ (resp. $\supp_{C^{1,1}}u$), is the set of $(t,x)\in(0,+\infty)\times \R^{n}$ such that there exists a neighborhood $V$ of $(t,x)$ and the restriction of $u$ on $V$ is of class $C^1$. We call complement of $C^1$ support (resp. $C^{1,1}$ support) of $u$, denoted by $\singsupp_{C^1}u$ (resp. $\singsupp_{C^{1,1}}u$), the \emph{$C^1$ singular support} (resp. \emph{$C^{1,1}$ singular support}) of $u$.

In this section, we will show that the $\singsupp_{C^1}u$ propagates along the generalized characteristics globally for the contact type Hamilton-Jacobi equation \eqref{eq:HJe} (See \cite{Albano2014_1} for the classical case).

\begin{Lem}\label{glo l1}
$\singsupp_{C^1}u=\bar{\Sigma}$ and $\supp_{C^1}u=\supp_{C^{1,1}}u$.
\end{Lem}

\begin{proof}
By the definition of $\supp_{C^1}u$, we know that $\supp_{C^1}u\cap\bar{\Sigma}=\varnothing$. Since $u$ is locally semiconcave with linear modulus, we have $(0,+\infty)\times\R^{n}\setminus \bar{\Sigma}\subset\supp_{C^1}u$. Thus, $\supp_{C^1}u=(0,+\infty)\times\R^{n}\setminus \bar{\Sigma}$, i.e., $\singsupp_{C^1}u=\bar{\Sigma}$.
The relation $\supp_{C^{1,1}}u\subset\supp_{C^{1}}u$ is trivial. Since $u$ is a $C^1$ solution of (HJe) on $\supp_{C^{1}}u$, we conclude $u$ is locally semiconvex with linear modulus on $\supp_{C^{1}}u$. Therefore, $u$ is of class $C^{1,1}_{loc}$ on $\supp_{C^{1}}u$. It follows $\supp_{C^{1}}u\subset\supp_{C^{1,1}}u$ and $\supp_{C^1}u=\supp_{C^{1,1}}u$.
\end{proof}

\begin{Lem}\label{glo l2}
Suppose $(t,x_{i})\in (0,+\infty)\times\R^{n}$ and $\xi_{i}$ is a minimizer for $u(t,x_{i})$, $i=1,2$, $x_{1}\neq x_{2}$. Then, we have $\xi_{1}(s)\neq\xi_{2}(s)$ for any $s\in(0,t)$.
\end{Lem}

\begin{proof}
If there exists $s_{0}\in (0,t)$ such that $\xi_{1}(s_0)=\xi_{2}(s_0)$, let
\begin{equation*}
\xi(s)=
\begin{cases}
\xi_{1}(s),\ s\in[0,s_{0}]\\
\xi_{2}(s),\ s\in[s_{0},t].
\end{cases}
\end{equation*}
Then $\xi$ is a minimizer for $u(t,x_2)$. So $\xi$ is the solution of the associated Lie equation \eqref{eq:Lie2} with initial condition
\begin{equation*}
\begin{cases}
\xi(0)=\xi_{1}(0)\\
p(0)=L_{v}(0,\xi_{1}(0),\dot{\xi_{1}}(0),u_{0}(\xi_{1}(0)))\\
u(0)=u_{0}(\xi_{1}(0)).
\end{cases}
\end{equation*}
Then, Cauchy-Lipschitz theorem implies $\xi\equiv\xi_{1}$ since two two curves satisfy the same initial condition. Thus, $x_{2}=\xi_{2}(t)=\xi(t)=\xi_{1}(t)=x_{1}$. This leads to a contradiction with $x_{1}\neq x_{2}$. Therefore, $\xi_{1}(s)\neq\xi_{2}(s)$ for any $s\in(0,t)$.
\end{proof}

For $(t,x)\in\supp_{C^1}u$, let $y(\cdot;t,x)$ be the solution of the associated Lie equation \eqref{eq:Lie2} with terminal condition
\begin{equation*}
\begin{cases}
y(t)=x\\
p(t)=\nabla u(t,x)\\
u(t)=u(t,x).
\end{cases}
\end{equation*}
Then $y(\cdot;t,x)$ is the unique minimizer for $u(t,x)$. By Lemma \ref{glo l1} and Cauchy-Lipschitz theorem, we know that $y(s;t,x)$ is locally Lipschitz with respect to $(s;t,x)$.

\begin{Lem}\label{glo l3}
Suppose $(t_{0},x_{0})\in\supp_{C^1}u$. Then $(s,y(s;t_{0},x_{0}))\in\supp_{C^1}u$ for any $s\in(0,t_0)$.
\end{Lem}

\begin{proof}
Fix $s\in(0,t_0)$. $(t_0,x_0)\in\supp_{C^1}u$ implies that there exists $0<\varepsilon<s$ such that $B_{\varepsilon}(t_{0},x_{0})\subset\supp_{C^1}u$. Consider the map
\begin{align*}
  F:B_{\varepsilon}(t_{0},x_{0})\to&\,(0,+\infty)\times\R^{n},\\
    (t,x)\mapsto&\,(t-(t_{0}-s),y(t-(t_{0}-s);t,x)).
\end{align*}
Then, $F$ is a Lipschitz map and Lemma \ref{glo l2} implies $F$ is injective. Due to invariance of domain, we know that $F(B_{\varepsilon}(t_{0},x_{0}))$ is an open set. Noticing $F(B_{\varepsilon}(t_{0},x_{0}))\cap \Sigma=\varnothing$, we have $F(B_{\varepsilon}(t_{0},x_{0}))\cap \bar{\Sigma}=\varnothing$, that is, $F(B_{\varepsilon}(t_{0},x_{0}))\subset \supp_{C^1}u$. Therefore, we conclude $(s,y(s;t_{0},x_{0}))=F(t_{0},x_{0})\in\supp_{C^1}u$.
\end{proof}

\begin{The}\label{global propagation}
Suppose $(t_0,x_0)\in \bar{\Sigma}$, $T>t_{0}$, $\gamma:[t_{0},T]\to\R^{n}$ is a solution of of the differential inclusion
\begin{equation}\label{ci}
	\dot{x}\in \mbox{\rm co}H_{p}(s,x,\nabla^{+} u(s,x),u(s,x)),\ s\geqslant 0.
\end{equation}
with initial condition $\gamma(t_0)=x_0$. Then for any $t\in(t_0,T]$, we have $(t,\gamma(t))\in \bar{\Sigma}$.
\end{The}

\begin{proof}
If there exists $t\in (t_0,T]$ such that $(t,\gamma(t))\notin \bar{\Sigma}$. Then $(t,\gamma(t))\in \supp_{C^1}u$. By Lemma \ref{glo l3}, we know that $(s,y(s;t,\gamma(t)))\in \supp_{C^1}u$ for all $s\in [t_0,t]$. Thus, there exists an open set $V$ such that $\{(s,y(s;t,\gamma(t)))|s\in[t_0,t]\}\subset V \subset \supp_{C^1}u$. This implies $y(\cdot;t,\gamma(t))$ is the solution of the equation
\begin{equation}\label{ch}
	\dot{x}=H_{p}(s,x,\nabla u(s,x),u(s,x)),\ s\geqslant 0
\end{equation}
with terminal condition $y(t)=\gamma(t)$. Noticing that $\gamma(\cdot)$ is a solution of \eqref{ci} with the same terminal condition and \eqref{ci} coincides with \eqref{ch} in $V$, we have $\gamma(\cdot)=y(\cdot;t,\gamma(t))$ by Cauchy-Lipschitz theorem. Therefore, $(t_0,x_0)=\gamma(t_0)=y(t_0;t,\gamma(t))\in\supp_{C^1}u$. This leads to a contradiction since $(t_0,x_0)\in \bar{\Sigma}$. 
\end{proof}


\section{Proofs of the propositions in Section \ref{section_3_1}}\label{sec:App_A}

In this section we afford detailed proofs of all propositions in Section \ref{section_3_1} for completion.

\begin{proof}[Proof of Proposition \ref{pro:Herglotz_Lie}]
By Lemma 3.1 in \cite{CCJWY2020} we conclude that the set $\mathcal{Z}_{t,x}$ is non-empty. For any $y_{t,x}\in \mathcal{Z}_{t,x}$ there exists a minimal curve $\xi\in\Gamma^{0,t}_{y_{t,x},x}$ for $h_L(0,t,y_{t,x},x,u_0(y_{t,x}))$ which is also a minimizer of \eqref{eq:cov}. The rest of the proofs of (1) and (2) is follows from the proof of Theorem 1 in \cite{CCWY2019}.

Now, we turn to prove (3). Since $\xi$ is a minimizer of $h(0,t,y_{t,x},x,u_{0}(y_{t,x}))$, we have
\begin{align*}
|\dot{\xi}(s)|\leqslant F\left(t,\frac{|x-y_{t,x}|}{t},|u_{0}(y_{t,x})|\right),\qquad \forall s\in[0,t],
\end{align*}
where $F$ is a locally bounded function. By Lemma 3.1 in \cite{CCJWY2020}, there holds $\frac{|x-y_{t,x}|}{t}\leqslant C_{1}(t,x)$, which implies $|u_{0}(y_{t,x})|\leqslant \max_{y\in B_{tC_{1}(t,x)}} |u_{0}(y)|$. Thus,
\begin{align*}
|\dot{\xi}(s)|\leqslant F(t,C_{1}(t,x),\quad\max_{y\in B_{tC_{1}(t,x)}} |u_{0}(y)|)=C_{2}(t,x),\qquad \forall s\in[0,t].
\end{align*}
Then it can be easily seen that there exists a locally bounded function $C(t,x)$ such that the estimates in (3) holds.
\end{proof}

\medskip

\begin{proof}[Proof of Proposition \ref{pro:dyn_prog}]
The proof is very similar to Proposition 3.3 in \cite{CCJWY2020}.
\end{proof}

\medskip

\begin{proof}[Proof of Proposition \ref{pro:D^*}]
The proof of (1) is very similar to Proposition 3.1 in \cite{CCJWY2020}, and The proof of (2) is directly from the semiconcavity estimate in the paper [?].

Now, Suppose $(t,x)\in (0,\infty)\times\R^{n}$. By semiconcavity of $u$, we have $D^{*}u(t,x)\subset D^{+}u(t,x)$ and $D^{+}u(t,x)=\mbox{\rm co}D^{*}u(t,x)$. This implies $\mbox{\rm Ext}(D^{+}u(t,x))\subset D^{*}u(t,x)$. It follows from (1) that
\begin{align*}
q+H(t,x,p,u(t,x))=0,\ \forall (q,p)\in D^{*}u(t,x).
\end{align*}
Thus, $D^*u(t,x)\subset\{(q,p)\in D^+u(t,x): q+H(t,x,p,u(t,x))=0\}$. To obtain the conclusion, we only need to show $\{(q,p)\in D^+u(t,x): q+H(t,x,p,u(t,x))=0\}\subset \mbox{\rm Ext}(D^{+}u(t,x))$.
For any $(q,p)\in D^{+}u(t,x)/\mbox{\rm Ext}(D^{+}u(t,x))$, there exists $(q_{1},p_{1}),(q_{2},p_{2})\in D^{+}u(t,x),\ (q_{i},p_{i})\neq(q,p),i=1,2$ and $0<\lambda<1$ such that $(q,p)=\lambda(q_{1},p_{1})+(1-\lambda)(q_{2},p_{2})$. If $p_{1}=p_{2}=p$, then there holds $q_{1}<q<q_{2}$ (or $q_{2}<q<q_{1}$) and
\begin{align*}
q+H(t,x,p,u(t,x))<q_{2}+H(t,x,p_{2},u(t,x))\leqslant 0.
\end{align*}
If $p,p_{1},p_{2}$ are all different, then by strictly convexity of $H$,
\begin{align*}
&q+H(t,x,p,u(t,x))=\lambda q_{1}+(1-\lambda)q_{2}+H(t,x,\lambda p_{1}+(1-\lambda)p_{2},u(t,x))\\
&<\lambda(q_{1}+H(t,x,p_{1},u(t,x)))+(1-\lambda)(q_{2}+H(t,x,p_{2},u(t,x)))\leqslant 0.
\end{align*}
So we have
\begin{align*}
q+H(t,x,p,u(t,x))<0,\quad \forall (q,p)\in D^{+}u(t,x)\setminus\mbox{\rm Ext}(D^{+}u(t,x)).
\end{align*}
Therefore, $\{(q,p)\in D^+u(t,x): q+H(t,x,p,u(t,x))=0\}\subset \mbox{\rm Ext}(D^{+}u(t,x))$. This completes the proof of (3).

Finally, we turn to the proof of (4). Set $0<s<t$. By dynamic programming principle, we know that
\begin{align*}
u(\tau,\xi(\tau))=u(s,\xi(s))+\int_{s}^{\tau} L(r,\xi(r),\dot{\xi}(r),u_{\xi}(r)) dr,\quad \forall \tau\in[s,t],
\end{align*}
where $u_{\xi}$ is uniquely defined by \eqref{eq:cara1} with initial condition $u_{\xi}(s)=u(s,\xi(s))$. For any $(q,p)\in D^{+}u(s,\xi(s))$, we have
\begin{align*}
\lim_{\tau\to s^{+}}\frac{\int_{s}^{\tau} L(r,\xi(r),\dot{\xi}(r),u_{\xi}(r)) dr}{\tau-s}=L(s,\xi(s),\dot{\xi}(s),u_{\xi}(s))\geqslant p\cdot\dot{\xi}(s)-H(s,\xi(s),p,u(s,\xi(s))),
\end{align*}
and
\begin{align*}
          &\limsup_{\tau\to s^{+}}\frac{u(\tau,\xi(\tau))-u(s,\xi(s))}{\tau-s}=\limsup_{\tau\to s^{+}}\frac{u(s+(\tau-s),\xi(s)+\dot{\xi}(s)(\tau-s))-u(s,\xi(s))}{\tau-s}\\
\leqslant &\langle (q,p),(1,\dot{\xi}(s))\rangle =q+p\cdot\dot{\xi}(s).
\end{align*}
It follows that $q+H(s,\xi(s),p,u(s,\xi(s)))\geqslant0$. Combing this with (1), we obtain
\begin{align*}
q+H(s,\xi(s),p,u(s,\xi(s)))=0,\qquad \forall (q,p)\in D^{+}u(s,\xi(s)).
\end{align*}
Thus, by (3) we have $D^{+}u(s,\xi(s))=\mbox{\rm Ext}(D^{+}u(s,\xi(s)))$. Therefore, $D^{+}u(s,\xi(s))$ is a singleton and $u$ is differentiable at $(s,\xi(s))$.
\end{proof}

\medskip

\begin{proof}[Proof of Proposition \ref{pro:sensitive}]
For any $v\in\R^{n},|v|\leqslant1$, let $\xi_{v}(s)=\xi(s)+\frac{s}{t}v,s\in[0,t]$ and $\tilde{u}_{v}$ be the solution of
\begin{equation}\label{pf 3.7 1}
\begin{cases}
\dot{\tilde{u}}_{v}=L(s,\xi_{v}(s),\dot{\xi}_{v}(s),\tilde{u}_{v}(s))=L(s,\xi(s)+\frac{s}{t}v,\dot{\xi}(s)+\frac{v}{t},\tilde{u}_{v}(s)),\ s\in[0,t],\\
\tilde{u}_{v}(0)=u_{0}(\xi_{v}(0))=u_{0}(\xi(0)).
\end{cases}
\end{equation}
Differentiate \eqref{pf 3.7 1} with respect to $v$, we have
\begin{align*}
\begin{cases}
\frac{\partial}{\partial v}\dot{\tilde{u}}_{v}=L_{x}\cdot\frac{s}{t}+L_{v}\cdot\frac{1}{t}+L_{u}\cdot\frac{\partial}{\partial v}\tilde{u}_{v},\ s\in[0,t]\\
\frac{\partial}{\partial v}\tilde{u}_{v}(0)=0.
\end{cases}
\end{align*}
Solving this Cauchy problem at $v=0$, we obtain that
\begin{align*}
\frac{\partial}{\partial v}\Big\vert_{v=0}\tilde{u}_{v}(t)&=e^{\int_{0}^{t}L_{u}dr}\int_{0}^{t}e^{-\int_{0}^{s}L_{u}dr}(L_{x}\cdot\frac{s}{t}+L_{v}\cdot\frac{1}{t})ds\\
                                   &=e^{\int_{0}^{t}L_{u}dr}[e^{-\int_{0}^{s}L_{u}dr}L_{v}\cdot\frac{s}{t}\Big\vert_{s=0}^{s=t}
                                   +\int_{0}^{t}(e^{-\int_{0}^{s}L_{u}dr}L_{x}-\frac{d}{ds}e^{-\int_{0}^{s}L_{u}dr}L_{v})\frac{s}{t}ds]\\
                                   &=e^{\int_{0}^{t}L_{u}dr}\cdot e^{-\int_{0}^{t}L_{u}dr}\cdot L_{v}(t,\xi(t),\dot{\xi}(t),u_{\xi}(t))\cdot 1\\
                                   &=L_{v}(t,\xi(t),\dot{\xi}(t),u_{\xi}(t))=p(t).
\end{align*}
Therefore,
\begin{align*}
&\limsup_{v\to 0}\frac{u(t,x+v)-u(t,x)-p(t)\cdot v}{|v|}\leqslant\limsup_{v\to 0}\frac{\tilde{u}_{v}(t)-u(t,x)-p(t)\cdot v}{|v|}\\
&=\lim_{v\to 0}\frac{\tilde{u}_{v}(t)-\tilde{u}_{0}(t)-p(t)\cdot v}{|v|}=0,
\end{align*}
which implies $p(t)\in \nabla^{+}u(t,x)$. This completes the proof of the first relation.
Notice that the second relation is a direct consequence of the first one and Proposition \ref{pro:D^*} (4).
Now, we turn to the proof of the third relation. For any $v\in\R^{n},|v|\leqslant1$, let $\xi_{v}(s)=\xi(s)+\frac{t-s}{t}v,s\in[0,t]$ and $\tilde{u}_{v}$ be the solution of
\begin{equation}\label{pf 3.7 2}
\begin{cases}
\dot{\tilde{u}}_{v}=L(s,\xi_{v}(s),\dot{\xi}_{v}(s),\tilde{u}_{v}(s))=L(s,\xi(s)+\frac{t-s}{t}v,\dot{\xi}(s)-\frac{v}{t},\tilde{u}_{v}(s)),\ s\in[0,t],\\
\tilde{u}_{v}(0)=u_{0}(\xi_{v}(0))=u_{0}(\xi(0)+v).
\end{cases}
\end{equation}
Differentiating \eqref{pf 3.7 2} with respect to $v$, we have
\begin{align*}
\begin{cases}
\frac{\partial}{\partial v}\dot{\tilde{u}}_{v}=L_{x}\cdot\frac{t-s}{t}+L_{v}\cdot(-\frac{1}{t})+L_{u}\cdot\frac{\partial}{\partial v}\tilde{u}_{v},\ s\in[0,t]\\
\frac{\partial}{\partial v}\tilde{u}_{v}(0)=Du_{0}(\xi(0)+v).
\end{cases}
\end{align*}
Solving this Cauchy problem at $v=0$, we obtain that
\begin{align*}
\frac{\partial}{\partial v}\Big\vert_{v=0}\tilde{u}_{v}(t)&=e^{\int_{0}^{t}L_{u}dr}[Du_{0}(\xi(0))+\int_{0}^{t}e^{-\int_{0}^{s}L_{u}dr}(L_{x}\cdot\frac{t-s}{t}+L_{v}\cdot(-\frac{1}{t}))ds]\\
                                   &=e^{\int_{0}^{t}L_{u}dr}[Du_{0}(\xi(0))+e^{-\int_{0}^{s}L_{u}dr}L_{v}\cdot\frac{t-s}{t}\Big\vert_{s=0}^{s=t}
                                   +\int_{0}^{t}(e^{-\int_{0}^{s}L_{u}dr}L_{x}-\frac{d}{ds}e^{-\int_{0}^{s}L_{u}dr}L_{v})\frac{t-s}{t}ds]\\
                                   &=e^{\int_{0}^{t}L_{u}dr}[Du_{0}(\xi(0))-L_{v}(0,\xi(0),\dot{\xi}(0),u_{\xi}(0))]\\
                                   &=e^{\int_{0}^{t}L_{u}dr}[Du_{0}(\xi(0))-p(0)].
\end{align*}
On the other hand,
\begin{align*}
\tilde{u}_{v}(t)\geqslant u(t,x)=\tilde{u}_{0}(t),\ \forall v\in\R^{n},|v|\leqslant1.
\end{align*}
This implies $\frac{\partial}{\partial v}\Big\vert_{v=0}\tilde{u}_{v}(t)=0$. Thus, $e^{\int_{0}^{t}L_{u}dr}[Du_{0}(\xi(0))-p(0)]=0$, that is $p(0)=Du_{0}(\xi(0))$.
\end{proof}

\medskip

\begin{proof}[Proof of Proposition \ref{pro:contact_minimizer}]
Now we have
\begin{align*}
\dot{U}&=P\cdot H_{p}(s,X,P,U)-H(s,X,P,U)=L(s,X,H_{p}(s,X,P,U),U)\\
       &=L(s,X,\dot{X},U),\ \forall s\in[0,t],
\end{align*}
and $U(0)=u_{0}(X(0))$, that is, $U$ is the solution of the Cauchy problem \eqref{eq:cara1} with $\xi=X$. Thus, $U(t)=u(t,X(t))$ implies $X$ is a minimizer of $u(t,X(t))$.
\end{proof}

\medskip

\begin{proof}[Proof of Proposition \ref{pro:1-1}]
If $u$ is differentiable at $(t,x)$, then by Proposition \ref{pro:Herglotz_Lie} and Proposition \ref{pro:contact_minimizer}, the solution $X$ of \eqref{eq:Lie2} with terminal condition
\begin{align*}
\begin{cases}
X(t)=x,\\
P(t)=\nabla u(t,x),\\
U(t)=u(t,x),
\end{cases}
\end{align*}
is the unique minimizer if $u(t,x)$. Thus, our conclusion obviously holds.

Now consider the case of a general $(t,x)$. Fix any $(q,p)\in D^{*}u(t,x)$, and let $(X,P,U)$ be the solution of \eqref{eq:Lie2} with terminal condition
\begin{align*}
\begin{cases}
X(t)=x,\\
P(t)=p,\\
U(t)=u(t,x).
\end{cases}
\end{align*}
Then there exists a sequence of points $(t_{i},x_{i})$ where $u$ is differentiable which converges to $(t,x)$ such that $\lim_{i\to\infty}Du(t_{i},x_{i})=(q,p)$. Invoking the first step, we know that the solution $X_{i}$ of \eqref{eq:Lie2} with terminal condition
\begin{align*}
\begin{cases}
X_{i}(t_{i})=x_{i},\\
P_{i}(t_{i})=\nabla u(t_{i},x_{i}),\\
U_{i}(t_{i})=u(t_{i},x_{i}),
\end{cases}
\end{align*}
is the unique minimizer if $u(t,x)$. It follows that $U_{i}(0)=u_{0}(X_{i}(0))$ and $(X_{i},P_{i},U_{i})$ converges to $(X,P,U)$ in $C^{R+1}$ topology. Therefore, we have
\begin{align*}
U(0)=\lim_{i\to\infty}U_{i}(0)=\lim_{i\to\infty}u_{0}(X_{i}(0))=u_{0}(X(0)),
\end{align*}
which implies $X$ is a minimizer of $u(t,x)$ by Proposition \ref{pro:contact_minimizer}. This proves that the correspondence $(q,p)\to X$ is a map from $D^{*}u(t,x)$ to the set of minimizers for $u(t,x)$.
Next, we prove the injectivity of our map. If $(q_{i},p_{i})\in D^{*}u(t,x),i=1,2$ map to the same minimizer $X$ for $u(t,x)$, then by Proposition \ref{pro:sensitive} and Proposition \ref{pro:D^*} we have
\begin{align*}
p_{i}=P_{i}(t)=L_{v}(t,X(t),\dot{X}(t),u(t,x)):=p,\ i=1,2,
\end{align*}
and
\begin{align*}
q_{i}=-H(t,x,p,u(t,x)),\ i=1,2.
\end{align*}
Thus, $(q_{1},p_{1})=(q_{2},p_{2})$.
Finally, we prove the surjectivity. Let $X:[0,t]\to\R^{n}$ be a minimizer for $u(t,x)$ and $P(s)=L_{v}(s,X(s),\dot{X}(s),u(s,X(s))),\ s\in[0,t]$. By Proposition \ref{pro:D^*} and Proposition \ref{pro:sensitive}, we know that for any $s\in(0,t)$, $Du(s,X(s))$ exists, and we have
\begin{align*}
\nabla u(s,X(s))=P(s),\ u_{t}(s,X(s))=-H(s,X(s),P(s),u(s,X(s))).
\end{align*}
It follows that
\begin{align*}
\lim_{s\to t^{-}}\nabla u(s,X(s))=P(t),\ \lim_{s\to t^{-}}u_{t}(s,X(s))=-H(t,x,P(t),u(t,x)).
\end{align*}
Therefore, $(-H(t,x,P(t),u(t,x)),P(t))\in D^{*}u(t,x)$.
\end{proof}

\medskip

\begin{proof}[Proof of Proposition \ref{pro:approx}]
By Proposition \ref{pro:1-1}, there exists $(q_{k},p_{k})\in D^{*}u(t_{k},x_{k})$ such that $\xi_{k}=X_{k}$, where $(X_{k},P_{k},U_{k})$ is the solution of \eqref{eq:Lie2} with terminal condition
\begin{align*}
\begin{cases}
X_{k}(t_{k})=x_{k},\\
P_{k}(t_{k})=p_{k},\\
U_{k}(t_{k})=u(t_{k},x_{k}).
\end{cases}
\end{align*}
Since $u$ is locally Lipschitz, we can choose a sub-sequence $\{(q_{k_{i}},p_{k_{i}})\}$ of $\{(q_{k},p_{k})\}$ such that $(q_{k_{i}},p_{k_{i}})\to(q,p)\in D^{*}u(t,x)$ as $i\to\infty$. Therefore, $(X_{k_{i}},P_{k_{i}},U_{k_{i}})$ converges in the $C^{R+1}$ topology to $(X,P,U)$, which is the solution of \eqref{eq:Lie2} with terminal condition
\begin{align*}
\begin{cases}
X(t)=x,\\
P(t)=p,\\
U(t)=u(t,x).
\end{cases}
\end{align*}
Again by Proposition \ref{pro:1-1}, we know that $X$ is a minimizer for $u(t,x)$. This completes the proof of (1), and (2) is a direct consequence of (1).
\end{proof}

\bibliographystyle{plain}
\bibliography{mybib}

\end{document}